\documentclass[11pt]{amsart}
\usepackage[english]{babel}
\usepackage[margin=1.6in]{geometry}
\usepackage{amsmath}
\usepackage{amsfonts}
\usepackage{amssymb}
\usepackage{amsthm}

\usepackage[english]{babel}
\usepackage{yfonts}
\usepackage[T1]{fontenc}
\usepackage[utf8x]{inputenc}
\usepackage{enumerate}
\usepackage{enumitem}
\usepackage{verbatim}
\usepackage{graphicx}
\usepackage{verbatim}
\usepackage{faktor}
\usepackage{xcolor}
\usepackage{xfrac}
\usepackage{tikz,tikz-cd}
\usepackage[all]{xy}
\usepackage{hyperref}

\usepackage{calrsfs}
\DeclareMathAlphabet{\pazocal}{OMS}{zplm}{m}{n}

\newcommand{\M}[4]{\overline{\pazocal M}_{#1,#2}(#3,#4)}
\newcommand{\Mone}[3]{\overline{\pazocal M}^{(1)}_{#1}(#2,#3)}
\newcommand{\PP}{\mathbb P}
\renewcommand{\k}{\mathbf k}
\newcommand{\OO}{\mathcal O}
\renewcommand{\to}{\rightarrow}
\newcommand{\pr}{\rm{pr}}
\newcommand{\Aaff}{\mathbb A}
\newcommand{\N}{\mathbb N}
\newcommand{\A}{\pazocal A}
\newcommand{\B}{\pazocal B}
\newcommand{\X}{\pazocal X}
\newcommand{\XP}{\pazocal{X\!P}}
\newcommand{\tXP}{\widetilde{\pazocal{X\!P}}}
\renewcommand{\L}{\mathcal L}
\newcommand{\hL}{\widehat{\mathcal L}}

\newcommand{\cC}{\mathcal C}
\newcommand{\hC}{\widehat{\mathcal C}}
\newcommand{\Z}{\pazocal Z}
\newcommand{\Zp}{\pazocal Z^p}
\newcommand{\tZ}{\widetilde{\pazocal Z}}
\newcommand{\tZp}{\widetilde{\pazocal Z}^p}
\newcommand{\oZp}{\overline{\Z}^{p,\rm{gst}}}
\newcommand{\MM}{\mathfrak M}
\newcommand{\pP}{\mathfrak P}
\newcommand{\hP}{\widehat{\mathfrak P}}
\newcommand{\hM}{\widehat{\mathfrak M}}
\newcommand{\hrM}{\widehat{\rm{M}}}
\newcommand{\oM}{\overline{\pazocal M}}
\newcommand{\tM}{\widetilde{\pazocal M}}
\newcommand{\R}{\operatorname{R}}
\newcommand{\Gm}{\mathbb{G}_{\rm{m}}}
\newcommand{\w}{\mathbf{w}}
\newcommand{\PtoM}{\lambda}
\newcommand{\vir}[1]{[#1]^{\rm{vir}}}
\newcommand{\virloc}[1]{[#1]^{\rm{vir}}_{\rm{loc}}}
\newcommand{\dvr}{\Delta}

\newcommand{\bq}{\begin{equation}}
\newcommand{\eq}{\end{equation}}
\newcommand{\ba}{\begin{aligned}}
\newcommand{\ea}{\end{aligned}}
\newcommand{\be}{\begin{enumerate}}
\newcommand{\ee}{\end{enumerate}}
\newcommand{\bsm}{\left(\begin{smallmatrix}}
\newcommand{\esm}{\end{smallmatrix}\right)}                   
\newcommand{\bpm}{\begin{pmatrix}}
\newcommand{\epm}{\end{pmatrix}}
\newcommand{\barr}{\begin{displaymath}\begin{array}{cccc}}
\newcommand{\earr}{\end{array}\end{displaymath}}
\newcommand{\barrl}{\begin{displaymath}\begin{array}{lcl}}
\newcommand{\earrl}{\end{array}\end{displaymath}}
\newcommand{\barl}{\begin{displaymath}\begin{array}{l}}
\newcommand{\earl}{\end{array}\end{displaymath}}
\newcommand{\bxym}{ \begin{displaymath}\xymatrix }
\newcommand{\exym}{\end{displaymath}}
\newcommand{\bcd}{\begin{center}\begin{tikzcd}}
\newcommand{\ecd}{\end{tikzcd}\end{center}}

\newcommand{\Spec}{\underline{\operatorname{Spec}}}

\newcommand{\Pic}{\operatorname{Pic}}
\newcommand{\Hom}{\operatorname{Hom}}
\newcommand{\id}{{\rm id}}

\theoremstyle{plain}
\newtheorem{thm}{Theorem}[section]
\newtheorem{conj}{Conjecture}
\newtheorem{lem}[thm]{Lemma}
\newtheorem{prop}[thm]{Proposition}

\newtheorem{cor}[thm]{Corollary}
\newtheorem*{teo*}{Theorem}
\newtheorem*{matteo*}{Main Theorem}

\newtheorem{claim}{Claim}

\theoremstyle{definition}

\newtheorem{ex}[thm]{Example}
\newtheorem{dfn}[thm]{Definition}
\newtheorem{remark}[thm]{Remark}

\newcommand{\thismonth}{\ifcase\month 
  \or January\or February\or March\or April\or May\or June%
  \or July\or August\or September\or October\or November%
  \or December\fi}

\title{Reduced invariants from cuspidal maps}
\author{L.Battistella, F.Carocci, C.Manolache}
\date{\today}

\begin{document}

\begin{abstract}

We consider genus 1 enumerative invariants arising from the Smyth--Viscardi moduli space of stable maps from curves with nodes and cusps. We prove that these invariants are equal to the Vakil--Zinger reduced invariants for the quintic threefold, providing a modular interpretation of the latter.

\end{abstract}

\maketitle
\tableofcontents

\section{Introduction}

While genus $0$ Gromov--Witten theory is understood fairly well, the higher genus theory is still unknown in general and its enumerative meaning is affected by more degenerate contributions. The Vakil--Zinger \emph{reduced} Gromov--Witten invariants solve two problems at the same time: they give computations of genus $1$ Gromov--Witten invariants and they have a better enumerative meaning. The only unsatisfactory feature in Zinger's approach is that the modular interpretation is unclear. In this paper we propose a way to fix this issue by working with a different moduli space. 
\subsection{Reduced invariants} We briefly recall the definition and the main results which concern reduced Gromov-Witten invariants. The locus of maps from smooth elliptic curves to $\PP^r$ is irreducible; we call its closure  the \emph{main component} of $\M{1}{n}{\PP^r}{d}$. In \cite{VZ} R. Vakil and A. Zinger construct a desingularisation $\widetilde{\pazocal M}_{1,n}(\PP^r,d)^{\rm{main}}$ of the main component via an iterated blow-up construction. Let
\[\xymatrix{\widetilde{\mathcal C}\ar[r]^{\tilde{f}}\ar[d]_{\tilde{\pi}}&\PP^r\\
\widetilde{\pazocal M}_{1,n}(\PP^r,d)^{\rm{main}}}
\]
be the universal curve over $\widetilde{\pazocal M}_{1,n}(\PP^r,d)^{\rm{main}}$ and let $l\in\N$. Then, the sheaf 
$\tilde\pi_*\tilde f^*\mathcal {O}_{\PP^r}(l)$  is a vector bundle. Given $X_l$ a hypersurface of degree $l$ in $\PP^r$, Vakil--Zinger define genus $1$ \emph{reduced invariants} of $X_l$ as:
\[ c_{\rm{top}}(\tilde\pi_*\tilde f^*(\mathcal O_{\PP^r}(l)))\cap[\widetilde{\pazocal M}_{1,n}(\PP^r,d)^{\rm{main}}].\]

Standard and reduced invariants are related by the Li--Zinger formula \cite{LZ}:
\begin{equation}\label{lizinger}
N_1(X_5,d)=N_1^{\rm{red}}(X_5,d)+\frac{1}{12}N_0(X_5,d).
\end{equation}

\subsection{Cuspidal Gromov--Witten invariants} \label{alternatecompactifications}

Based on D.I. Smyth's work on the birational geometry of $\oM_{1,n}$ \cite{SMY1,SMY2}, M. Viscardi \cite{VISC} has introduced a series of \emph{alternate compactifications} of the moduli space of maps from smooth elliptic curves. The first instance of these alternate compactifications is $\oM_{1,n}^{(1)}(X,\beta)$, where elliptic tails are made unstable and cuspidal singularities are allowed in the source curve. These spaces carry perfect obstruction theories and lead to invariants; we call them cuspidal Gromov--Witten invariants. 
We conjecture the following analogue of the Li--Zinger formula for cuspidal invariants.
\begin{conj}\label{conj}
Let $X$ be a smooth projective threefold, $\gamma_i\in A^*(X)$. Reduced invariants with insertions $\gamma_i$ are equal to cuspidal invariants with insertions $\gamma_i$.
\end{conj}

The main result of this paper is that Conjecture \ref{conj} holds for the quintic threefold. More precisely, we have the following.

\begin{matteo*}\label{thm:main}
Let $X_5\subseteq\PP^4$ be a generic smooth quintic threefold. Then:
\[
N_1^{\rm{red}}(X_5,d)=N_1^{\rm{cusp}}(X_5,d),
\]
where:
\[N_1^{\rm{red}}(X_5,d):=\deg \left(c_{\rm{top}}(\tilde\pi_*\tilde f^*(\mathcal O_{\PP^4}(5)))\cap [\widetilde{\pazocal M}_{1}(\PP^4,d)^{\rm{main}}]\right)\]
and
\[N_1^{\rm{cusp}}(X_5,d):=\deg [\Mone{1}{X_5}{d}]^{\rm{vir}}.
\]
\end{matteo*}

\subsection{The proof of the main Theorem}
The moduli space $\oM_1(\PP^r,d)$ has many boundary components. From the proof of the Li--Zinger formula we see that the genus $0$ contribution to the right hand side of (\ref{lizinger}) comes from the boundary component with a single rational tail. This suggests that discarding this component would provide a more direct approach to reduced invariants for the quintic threefold. This is exactly what the Viscardi space $\Mone{1}{\PP^r}{d}$ does.

Here is a rough sketch of the proof of our result. We first show:
\begin{teo*}
There exists a well-defined $1$-stabilisation morphism at the level of weighted-stable curves:
\[
\MM_{1,n}^{\rm{wt,st}}\to \MM_{1,n}^{\rm{wt,st}}(1)
\]replacing elliptic tails of weight $0$ with cusps.
\end{teo*}
We then consider the following cartesian diagram:
\bcd
\Z_{X_5}\ar[r]\ar[d]\ar[rd,phantom,"\Box"] & \overline{\mathcal M}_1^{(1)}(X_5,d) \ar[d] \\
\MM_1^{\rm{wt}}\ar[r] & \MM_1^{\rm{wt}}(1).
\ecd

We prove that $\Z_{X_5}$ is a substack of $\oM_1(X_5,d)$ that has no component with contracted elliptic tails. The diagram above endows $\Z_{X_5}$ with a virtual class which, by Costello's virtual push-forward formula \cite{costello}, has the same enumerative content as $\vir{\Mone{1}{X_5}{d}}$. 

In order to compare the degree of $\vir{\Z_{X_5}}$ with Vakil-Zinger's reduced invariants we follow in Chang and Li's footsteps. We first introduce the moduli space of $1$-stable maps with $p$-fields; this has the advantage of having a simpler geometry and admits a cosection-localised virtual class with the same degree as $\vir{\Mone{1}{X_5}{d}}$, see \cite{CLpfields}. We then construct
\[\MM_{1}^{\rm{wt,st}}\times_{\MM_{1}^{\rm{wt,st}}(1)}\Mone{1}{\PP^4}{d}^p,\]
and perform a desingularisation of it via the study of local equations \cite{HL}. In the end we analyse a splitting of the intrinsic normal cone \cite{CL}. All these steps deliver the theorem.

\subsection{Smoothability} On our way there, we review Vakil's characterisation of maps lying in the boundary of $\oM_1(\PP^r,d)^{\rm{main}}$ and we give a new proof using Hu-Li's equations. We also review and extend Smyth's work on genus $1$ singularities, and rephrase smoothability as follows.

\begin{teo*}
A map $[f]\in\oM_1(\PP^r,d)$ contracting the minimal genus $1$ subcurve is smoothable if and only if it factors through a map from a genus $1$ singularity that does not contract any component.
\end{teo*}

\subsection{Relation to other works} We view our result as an analogue of the Li--Zinger formula for cuspidal invariants. The Li--Zinger formula was first proved by J. Li and A. Zinger \cite{LZ} in the symplectic category and with algebraic-geometric methods by H.-L. Chang and J. Li \cite{CL}. Our approach is an adaptation of \cite{CL}. 

Zinger has computed reduced invariants in \cite{zingred}. It would be interesting to see if there is a direct calculation for cuspidal invariants. Zinger's computation together with the Li--Zinger formula gives a computation of genus $1$ Gromov--Witten invariants \cite{zingerstvsred}. 

The reduced Gromov--Witten invariants are related to the Gopakumar--Vafa invariants, and they coincide for Fano targets. Indeed the Gopakumar--Vafa invariants are by definition related to Gromov--Witten invariants by a recursive formula which takes into account degenerate lower genus boundary contributions. These contributions were computed by R. Pandharipande in \cite {degcont}. We do not have reduced invariants for higher genus, but if we had a modular interpretation of the main component, we could view the Gopakumar--Vafa formula as a higher genus analogue of the Li--Zinger formula. 

Recently, D. Ranganathan, K. Santos-Parker, and J. Wise \cite{RSPW} have given a modular interpretation of the main component of the moduli space of stable maps via centrally aligned log structures and a factorisation property.

\subsection{Outline of the paper}
In \S\ref{sec:section1} we review some classical results about the irreducible components of $\oM_1(\PP^r,d)$ and local equations for this moduli stack in a smooth ambient space. We recall Smyth's $m$-elliptic singularities and Viscardi's alternate compactifications  of the space of maps. We discuss two different proofs of Vakil's characterisation of smoothable maps to projective space: one is based on Hu-Li's local equations; for the other one we give a classification of genus $1$ singularities.

In \S\ref{section:p-fields} we adapt Chang-Li's work on $p$-fields to our case: we introduce the moduli space $\Mone{1}{\PP^4}{d}^p$ of $1$-stable maps with $p$-fields, we endow it with a cosection-localised $0$-dimensional virtual class supported on a proper substack (depending on a homogeneous polynomial $\w$). We show that its degree coincides with the genus $1$ cuspidal invariants of the quintic threefold $X_5=V(\w)$ up to a sign.

In \S\ref{sec:comparison} we argue in two different ways that there is a morphism of Artin stacks
 \[\mathfrak M_{1,n}^{\operatorname{wt},\text{st}}\to\mathfrak M_{1,n}^{\operatorname{wt},\text{st}}(1)\]
extending the identity on smooth elliptic curves and replacing elliptic tails of weight $0$ by cusps. We then use it to define:
 \[\Z:= \mathfrak M_{1,n}^{\operatorname{wt},\text{st}}\times_{\mathfrak M_{1,n}^{\operatorname{wt},\text{st}}(1)} \oM^{(1)}_{1,n}(\PP^4,d),\quad \Zp:=\mathfrak M_{1,n}^{\operatorname{wt},\text{st}}\times_{\mathfrak M_{1,n}^{\operatorname{wt},\text{st}}(1)} \oM^{(1)}_{1,n}(\PP^4,d)^p.\]
 We show that $\Z$ is a closed substack of $\oM_1(\PP^4,d)$. Unfortunately we are not able to compare $\Zp$ directly with $\oM_1(\PP^4,d)^p$.
 
In \S\ref{sec:equations} and \S\ref{section:main_proof} we study local equations and a desingularisation of $\Zp$. With these in place, we adjust Chang-Li's splitting of the intrinsic normal cone to our case. This analysis allows us to prove the main theorem.

\subsection{Notations and conventions}
\begin{itemize}[leftmargin=0cm]
\item We work over an algebraically closed field $\k$ of characteristic $0$.

\item We fix a homogeneous polynomial $\mathbf w\in\k[x_0,\ldots,x_4]_5$ of degree $5$ such that $X_5=V(\mathbf w)\subseteq\PP^4$ is a generic smooth quintic threefold. We usually drop the subscript and write $X$ for $X_5$.

\item We fix a positive integer $d$ determining all the homology classes $\beta=d[\ell]\in A_1(\PP^4)$, line bundle weights, etc.

\item $\MM:=\MM_{1,n}^{\rm{wt,st}}$ denotes the moduli stack of prestable curves with a weight assignment subject to a stability condition, see \S\ref{sec:comparison}. The universal curve is $\pi\colon\cC\to \MM$.

\item $\pP:=\mathfrak{Pic}^{\rm{tot deg}=d,\rm{st}}_{1,n}$ denotes the Picard stack of $\pi\colon \cC\to \MM$ with universal line bundle $\mathcal{L}$ of total degree $d$, subject to the stability condition: \[\omega_\pi^{\rm{log}}\otimes\mathcal L^{\otimes 2}\ \text{is}\ \pi\text{-ample.}\]

\item $\MM_1^{\rm{div}}$ denotes the moduli space of nodal curves with a relative Cartier divisor satisfying a stability condition, see \S\ref{sec:section1}.

\item $\rm M:=\M{1}{n}{X}{\beta}$ denotes Kontsevich's moduli space of $n$-pointed genus $1$ stable maps to $X$ in the homology class $\beta$; we always denote by $(\pi,f)\colon\cC_{\rm M}\to \rm M\times X$ the universal curve and stable map.

\item Similarly we denote by $\widehat{\MM}:=\mathfrak M_{1,n}^{\rm{wt}=d,\rm{st}}(1)$ the stack of weighted-stable, at worst cuspidal curves (see \S\ref{sec:comparison}), with universal curve $\hat\pi\colon\hC\to \widehat{\MM}$.

\item Let $\hP$ denote the Picard stack of $\hat\pi\colon\hC\to\hM,$ with universal line bundle $\hL$ of total degree $d$ and satisfying the usual stability condition.

\item  Let $\MM^{\rm{div}}_1(1)$ be the moduli stack parametrising at worst cuspidal curves with a relative Cartier divisor subject to the usual stability condition.

\item $\widehat{\rm{M}}:=\Mone{1}{X}{\beta}$ is the Smyth-Viscardi's moduli space of $1$-stable maps;
we always denote by $(\hat{\pi},\hat f)\colon\hC_{\widehat{\rm M}}\to \widehat{\rm M}\times X$ the universal curve and stable map.

\item We usually work with unmarked curves, for we are interested in the Gromov-Witten theory of a Calabi-Yau threefold, so $n=0$.

\item We always denote the spectrum of a discrete valuation ring (DVR) by $\dvr$, with closed point $0$ and generic point $\eta$.

\item Subcurves are always connected.
 The \emph{minimal} arithmetic genus $1$ subcurve is called the \emph{core} or the \emph{circuit}.
\end{itemize}
\subsection*{Acknowledgements} This work was inspired by a discussion with Prof. R. Thomas. We thank Prof. T. Coates for useful discussions. L.B. and F.C. would like to thank F. Bernasconi, A. Dotto and Prof. D.I. Smyth for their clever suggestions.

L.B. is supported by a Royal Society 1st Year URF and DHF Research Grant Scheme.
C.M. is supported by an EPSRC funded Royal Society Dorothy Hodgkin Fellowship.
This work was supported by the Engineering and Physical Sciences Research Council grant EP/L015234/1: the EPSRC Centre for Doctoral Training in Geometry and Number Theory at the Interface.

\section{$\oM_1(\PP^r,d)$ - Components, equations  and alternate compactifications}\label{sec:section1}

\subsection{Local equations and components}
We start by recalling a description of the global geometry of $\oM_1(\PP^r,d)$. Besides the main component, which was defined in the Introduction, for every positive integer $k$ and partition $\lambda\vdash d$ into $k$ positive parts, there is an irreducible \emph{boundary component} $D_{\lambda}(\PP^r,d)$ defined to be the closure of the locus where:
\begin{enumerate}[label=(\roman*)]
\item the source curve is obtained by gluing a smooth elliptic curve $E$ with $k$  rational tails $R_i\cong\PP^1$,
\item the map contracts the elliptic curve $E$ to a point, 
\item the map has degree $\lambda_i$ on the rational tail $R_i$.
\end{enumerate}
Indeed $D_{\lambda}(\PP^r,d)$ is the image of the gluing morphism from the fiber product:
\[\left(\oM_{1,k}\times\prod_{i=1}^k\M{0}{1}{\PP^r}{\lambda_i}\right)\times_{(\PP^r)^k}\PP^r.\]
We will denote by $D^k$ the union of all the $D_{\lambda}(\PP^r,d)$ where $\lambda$ has $k$ parts.

\begin{prop}\label{prop:components}
\emph{(1)} These are all the irreducible components of $\oM_1(\PP^r,d)$: 
\[\oM_1(\PP^r,d)=\oM_1(\PP^r,d)^{\rm{main}}\cup\bigcup_{\lambda}D_{\lambda}(\PP^r,d).\]

\emph{(2)} A map $[f]$ lies in \emph{the boundary of the main component}  if and only if:
\begin{itemize}
\item either $f$ is non-constant on at least one irreducible component of the core,
\item or if $f$ contracts the core, write $C=E\ {}_{\mathbf p}\!\sqcup_{\mathbf q}\bigsqcup_{i=1}^k R_i$,  where $E$ is the \emph{maximal} contracted genus $1$ subcurve, then $\{\operatorname{d}\!f(T_{q_i}R_i)\}_{i=1}^k$ are \emph{linearly dependent} in $T_{f(E)}\PP^r$.
\end{itemize}
In this case we say that $[f]$ is smoothable.
\end{prop}
 
This is due to R. Vakil and A. Zinger, see \cite[Lemma~5.9]{Vre}\cite{VZpreview}. We shall later discuss a  proof of this fact based on local equations for the moduli space.

We now review Hu-Li's procedure for finding local equations of $\oM_1(\PP^r,d)$ in a smooth ambient space \cite{HL}. They will be useful when describing the structure of the intrinsic normal cone and its splitting.

Recall that a map $C\to\PP^r$ is given by a line bundle $L$ on $C$ together with $r+1$ sections in $H^0(C,L)$ that generate the line bundle. It is therefore natural to embed $\oM_1(\PP^r,d)$ as an open inside $\pi_*\mathcal L^{\oplus r+1}$ on the universal Picard stack $\mathfrak{P}$.

Hu and Li actually work over a local chart of $\mathfrak M_{1}^{\rm{div}}$, parametrising families of nodal curves with a relative Cartier divisor satisfying a stability condition, namely that $\omega_\pi\otimes\OO_{\cC}(2\pazocal D)$ is a $\pi$-ample line bundle, where $\pazocal D$ is the universal Cartier divisor. We can construct $\MM_{1,n}^{\rm{div}}$ as the open inside
\[C(\pi_*\mathcal L)=\Spec_\pP\operatorname{Sym}^{\bullet}(\R^1\pi_*\mathcal L^{\vee}\otimes\omega_{\pi})\]
(see \cite{CLpfields} and Section \ref{section:p-fields} below), where the section of $\mathcal L$ is not $0$ on any irreducible component of the curve. Alternatively this is the moduli functor of a prestable curve with a line bundle and a section up to scalar, which can be thought of as the hom-stack $\Hom_{\MM_1}(\cC,[\Aaff^1/\Gm])$; then we should pick the connected component where the line bundle has total degree $d$, and the open substack obtained by requiring weighted stability and the section not to vanish on any irreducible component of the curve. The morphism $\MM_{1,n}^{\rm{div}}\to\mathfrak P$ is obtained by declaring $\mathcal L_{\MM_{1,n}^{\rm{div}}}:=\OO_{\cC}(\pazocal D)$. 

Let $[f\colon C\to \PP^r]$ be a point of $\oM_1(\PP^r,d)$: we may fix homogeneous coordinates on $\PP^r$ in such a way that $D:=f^{-1}\{x_0=0\}$ is a simple divisor supported on the smooth locus of $C$, i.e. a $d$-uple of distinct smooth points; this property will then hold in a neighbourhood of $[f]$. This gives a morphism from an \'{e}tale chart $\pazocal U$ of $\oM_1(\PP^r,d)$ to an \'{e}tale chart $\pazocal V$ of $\mathfrak M_{1}^{\rm{div}}.$ We may assume that $\pazocal V$ is contained in the locus where the divisor consists of $d$ distinct smooth points; notice  that such locus is smooth over $\mathfrak M_{1}$. A map to $\PP^r$ shall now be thought of as a curve-divisor pair $(C,D)$ together with $r$ sections of $\mathcal O_C(D)$; the map can be written as $[1:u_1:\ldots:u_r]$, where $1$ denotes the given section of $\mathcal O_C\to\mathcal O_C(D)$.

Furthermore, \'{e}tale locally on $\oM_1(\PP^r,d)$, we may pick extra sections $\pazocal A$ and $\pazocal B$ of the universal curve $\cC\to\oM_1(\PP^r,d)$ such that:
\begin{enumerate}
\item they pass through the core;
\item they are distinct smooth points disjoint from the support of $\pazocal D$. 
\end{enumerate}
Now  $\pi_*\OO_{\cC}(\pazocal D+\pazocal A)$ is a vector bundle on $\pazocal V$ and $\pi_*\OO_{\cC}(\pazocal D)$ is carved inside it by the vanishing of the restriction map $\pi_*\OO_{\cC}(\pazocal D+\pazocal A)\to\pi_*\OO_{\cC}(\pazocal D+\pazocal A)\rvert_{\pazocal A}$.

\begin{prop}\label{prop:smoothability}
\emph{(1)} There is a locally closed embedding of an \'{e}tale chart of $\oM_1(\PP^r,d)$ inside the vector bundle $V:=\Spec_{\pazocal V}(\pi_*\OO_{\cC}(\pazocal D+\pazocal A)^{\oplus r})$:
\bcd
\oM_1(\PP^r,d) & \pazocal U\ar[l,"\acute{e}t"]\ar[r,hook]\ar[d] & V\ar[dl] \\
\MM_{1}^{\rm{div}} & \pazocal V\ar[l,"\acute{e}t"] &
\ecd
Notice that $V$ is smooth since $\pazocal V$ is.

\emph{(2)} Assume furthermore that $\pazocal V$ is affine; on $\pazocal V$ we then have: 
\[\pi_*\OO_{\cC}(\pazocal D+\A)\cong \pi_*\OO_{\cC}(\pazocal D+\A-\B)\oplus \pi_*\OO_{\cC}(\pazocal D+\A)_{|\B}\]
and restricting to $\A$ is zero on the second factor. 

\emph{(3)} Call $\varphi\colon \pi_*\OO_{\cC}(\pazocal D+\A-\B)\to\pi_*\OO_{\cC}(\pazocal D+\A)_{|\A}$ the map induced on the first factor by restricting to $\A$.  Locally on $\pazocal V$, $\pi_*\OO_{\cC}(\pazocal D+\A-\B)$ can be written as a sum of line bundles: let $D=\sum_{i=1}^d\delta_i$, then  
 \[\varphi=\oplus\varphi_i\colon\oplus_{i=1}^d \pi_*\mathcal O_{\cC}(\delta_i+\A-\B)\to \pi_*\OO_{\cC}(\pazocal D+\A)_{|\A}.\] 
  After choosing a trivialisation for each of the line bundles above,  $\varphi_i\colon\mathcal O_\pazocal V\to\mathcal O_\pazocal V$ is given by multiplication by \[\prod_{q\in[\A,\delta_i]}\zeta_q,\]
   where $\zeta_q$ is the smoothing coordinate on $\mathfrak M_1$ corresponding to the node $q$, and $[\A,\delta_i]$ denotes the set of nodes separating $\A$ (the core) from the point $\delta_i$.
\end{prop}
We may now choose coordinates $(w^j_i)^{j=1,\ldots,r}_{i,\ldots, d}$ on the fiber of $V\to\pazocal V$ compatible with the basis given in Proposition~\ref{prop:smoothability}~.(3) such that $\pazocal U$ is an open inside $\left\{F^1=\ldots=F^r=0\right\}\subseteq V$, where:
\[F^j=\sum_{i=1}^d{\left(\prod_{q\in[\A,\delta_i]}\zeta_q\right)w_i^j}.\]
We refer the reader to \cite[Lemma~4.10, Proposition~4.13]{HL} for the details; we are going to review the key ideas in \S\ref{sec:equations}, where we find local equations for $\Mone{1}{\PP^r}{d}$.

 We include here a proof of Proposition~\ref{prop:components}~.(2) based on Hu-Li's equations.

\begin{proof}
 Let us start with the easiest degenerate situation: a contracted elliptic curve attached to a $\PP^1$ of degree $d$ at a single node $q$. Equations for the moduli space of maps around such a point look like: 
 \[\zeta_q\sum_{i=1}^d w_i^j=0,\;\;\;\text{for} \;\;\;j=1,\ldots,r.\]

  Our point corresponds to a smoothable map if and only if the equations admit a solution with $\zeta_q\neq 0$, that is $\sum_{i=1}^d w_i^j=0$ for every $j$. Taking a coordinate $z$ on  $\PP^1$ centred at the node $q$,   the $i$-th basis vector corresponds to a polynomial vanishing at $q$ and at $\delta_l,\;\forall l\neq i$. This can be written as:
  \[e_i(z)=z\prod_{l\neq i}\frac{(z-\delta_l)}{-\delta_l},\] 
  where we have chosen a convenient normalisation. So the restriction to the rational tail of the map corresponding to the point of coordinates $(w_i^j)_{i=1,\ldots,d}^{j=1,\ldots,r}$ can be written as:
  \[[1:\sum_{i=1}^d w_i^1e_i(z):\ldots:\sum_{i=1}^d w_i^re_i(z)].\] 
  Differentiating with respect to $z$ we see that the image of the tangent vector at $q$ is given in affine coordinates around $f(E)$ by:
  \[(\sum_{i=1}^d w_i^1,\ldots,\sum_{i=1}^d w_i^r).\] 
  Hence smoothability is equivalent to the image of the tangent vector being zero.
 
 More generally we may assume that the dual graph is terminally weighted \cite[\S 3.1]{HL}. Assume there are $k$  rational tails of positive weight $R_h,\ h=1,\ldots,k$. Denote by $D(h)$ the set of indices $i$ such that $\delta_i$ belongs to the $R_h$, and by $E(h)$ the set of nodes separating the core from $R_h$. The equations will then take the following form:
 \[\sum_{h=1}^k\left(\prod_{q\in E(h)}\zeta_q\right)\left(\sum_{i\in D(h)}w_i^j\right)=0,\ j=1,\ldots,r,\]
 which can be assembled in matrix form as follows:
 $$W\cdot\underline\zeta:=\left(\sum_{i\in D(h)}w_i^j\right)_{j,h}\cdot\left(\prod_{q\in E(h)}\zeta_q\right)_h=0.$$
 We see that smoothability is equivalent to the linear dependence of the rows of the above matrix $W$. On the other hand we can choose a suitable coordinate $z_h$ around the node $q_h$ on $R_h$ and write the map as: 
 \[[1:p_h^1(z_h):\ldots:p_h^r(z_h)],\]
  where: 
 \[p_h^j(z_h)=\sum_{i\in D(h)}w_i^je_i^h(z_h) \quad \text{and} \quad e_i^h(z_h)=z_h\prod_{l\in D(h)\setminus\{i\}}\frac{(z_h-\delta_l)}{-\delta_l}.\] The elliptic curve is contracted to the point $[1:0:\ldots:0]$ and the tangent vector to $R_h$ at $q_h$ is mapped to the $h$-th row of $W$ (in affine coordinates around $f(E)$). Again we see that the map is smoothable if and only if the image of the tangent vectors to the rational tails at the nodes are linearly dependent in $T_{f(E)}\PP^r$.
 \end{proof}

\subsection{Smyth-Viscardi's compactifications}
 The moduli spaces of $m$-stable maps give alternate compactifications of $\pazocal M_{1,n}(X,\beta)$, parametrising maps from smooth elliptic curves.
\begin{dfn} Let $C$ be a connected, reduced, projective curve of arithmetic genus $1$ over $\k$, and let $p_1,\dots,p_n\in C$ be smooth, distinct points. A map $f\colon C\to X$ is said to be \emph{$m$-stable} if the following conditions hold:
\begin{enumerate}
 \item $C$ has only nodes and elliptic $l$-fold points, $l\leq m$, as singularities.
 \item For any subcurve $E\subset C$ of arithmetic genus $1$ on which $f$ is constant, 
 \[\left|\left\{E\cap \overline{C\setminus E}\right\}\cup\left\{i : p_i \in E\right\}\right| > m.\]
 \item $f$ has no non-trivial infinitesimal automorphisms.
 \end{enumerate}
\end{dfn}
Recall that a $\k$-rational $p\in C$ is called an \emph{elliptic $m$-fold point} if:
\[
\hat{\mathcal{O}}_{C,p}\cong
\begin{cases}
\k[[x,y]]/(y^2-x^3) &  m=1\\
\k[[x,y]]/(x(x-y^2))&  m=2\\
\k[[x,y]]/ I_m &  m\geq 3
\end{cases}
\] 
where $I_m=\left(x_hx_i-x_hx_j : i,j,h\in\left\{1,\dots,m-1\right\}\right)$ and $i,j,h$ are distinct.

 Viscardi's main result \cite[Theorem~3.6]{VISC} is the following:
\begin{thm}
The moduli functor of $m$-stable maps  $\overline{\pazocal{M}}^{(m)}_{1,n}(X,\beta)$ is represented by a proper Deligne-Mumford stack of finite type over $\k$.
\end{thm}

\begin{remark}
 The Behrend-Fantechi obstruction theory $\R^\bullet\hat\pi_*\hat f^* T_X$ for spaces of morphisms endows $\overline{\pazocal{M}}^{(m)}_{1,n}(X,\beta)$ with a perfect obstruction theory relative to $\MM_{1,n}(m)$. The base is an irreducible Artin stack for every $m$ and even a smooth stack for $m\leq 4$. We thus have a virtual class on $\overline{\pazocal{M}}^{(m)}_{1,n}(X,\beta)$ and $m$-stable invariants can be defined in the usual way. 
\end{remark}

\begin{remark}\label{remark:sprouting}
We think that the algorithm proposed by Viscardi to prove the properness of his moduli spaces oversees a case. The issue is that, given a map $f\colon \cC_\dvr\to\PP^r_\dvr$ over a DVR scheme $\dvr$ such that $\cC_\eta$ is smooth and $f_0$ is constant on a genus $1$ connected subcurve $E\subseteq \cC_0$, it is not always true that $f$ descends to a map $\hat f\colon \hC_\dvr\to\PP^r_\dvr$, where $\hC_0$ has a genus $1$ \emph{Gorenstein} singularity.

Consider the stable map $[f]$ in $\oM_{1}(\PP^3,4)$ from an elliptic bridge \[C:=R_1{}_{q_1}\!\sqcup E\sqcup_{q_2} R_2\] to $\PP^3$ that maps $R_1$ to the $z$-axis, contracts $E$ to the origin, and makes $(R_2,q_2)$ into the normalisation of a cusp in the $(x,y)$-plane, i.e. its image is the non-Gorenstein singularity:
\[D:=\operatorname{Spec}\left(\k[x,y,z]/(x,y)\right)\cup \operatorname{Spec}\left(\k[x,y,z]/(z,y^2-x^3)\right).\] 
Notice that $\operatorname{d}\!f(T_{q_2}R_2)=0$, so there is a non-trivial linear relation:
\[0\cdot\operatorname{d}\!f(T_{q_1}R_1)+1\cdot\operatorname{d}\!f(T_{q_2}R_2)=0,\]
and hence the map is smoothable. Viscardi claims that the map factors through the family $\hC_\dvr$, obtained by contracting $E$ to a tacnode. Notice that the image of $\hat f$ would still be $D$, since there is at most one indeterminacy point of $\hC_\dvr\dashrightarrow\PP^r_\dvr$, located at the singularity. However in our example $f$ cannot factor through the tacnode. Indeed in that case we would have a birational morphism between two singular curves with the same $\delta$ invariant and the same normalisation, so $\hat f\colon \widehat C\to D$ would be an isomorphism, which is a contradiction.
 We suggest that in this case the correct procedure would be to sprout $(R_1,q_1)$.
 
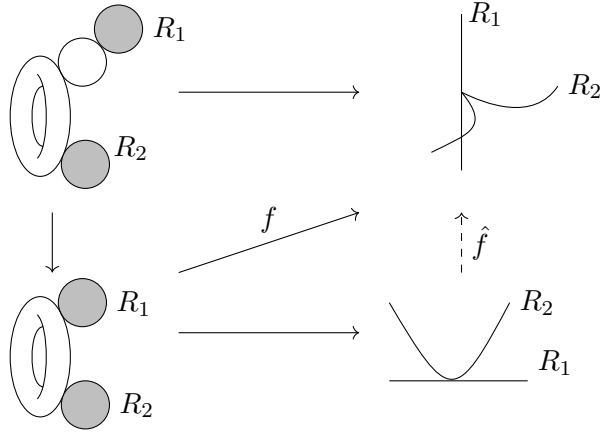
\begin{figure}[h]\label{fig:3-stab}
\centering
\begin{tikzpicture}[scale=.8]
\draw (-.1,0) circle (.4cm);
\draw[fill=gray!50] (.5,.55) circle (.4cm);
\node [right] at (.9,.55) {$R_1$};
\draw (-.8,-.9) ellipse (.5cm and 1cm);
\draw (-.85,-.2) .. controls (-.65,-.4) and (-.65,-1.4) .. (-.85,-1.6);
\draw (-.75,-.4) .. controls (-1,-.4) and (-1,-1.4) .. (-.75,-1.4);
\draw[fill=gray!50] (-.05,-1.7) circle (.4cm);
\node [above right] at (.25,-1.8) {$R_2$};

\draw[fill=gray!50] (-.1,-4) circle (.4cm);
\node [below right] at (.3,-3.65) {$R_1$};
\draw (-.8,-4.9) ellipse (.5cm and 1cm);
\draw (-.85,-4.2) .. controls (-.65,-4.4) and (-.65,-5.4) .. (-.85,-5.6);
\draw (-.75,-4.4) .. controls (-1,-4.4) and (-1,-5.4) .. (-.75,-5.4);
\draw[fill=gray!50] (-.05,-5.7) circle (.4cm);
\node [right] at (.35,-5.7) {$R_2$};

\draw (5,-5.3)-- (7.3,-5.3);
\node [ right] at (7.3,-5) {$R_1$};
\draw (5,-4) .. controls (6,-5.7) and (6.1,-5.7) .. (7,-4);
\node [ right] at (7,-4) {$R_2$};

\draw (6.2,-1.8)-- (6.2,.8);
\node at (6.6,.8) {$R_1$};
\draw (6.2,-.5) .. controls (6.6,-.7) and (7.4,-1) .. (7.8,-.4);
\node [right] at (7.8,-.4) {$R_2$};
\draw (6.2,-.5) .. controls (6.7,-1.1) and (6.3,-1.2) .. (5.7,.-1.5);

\draw[->] (1.5,-.5) -- (4.5,-.5);
\draw[->] (-.6,-2.5) -- (-.6,-3.5);
\draw[->] (1.5,-4.5) -- (4.5,-4.5);
\draw[->] (1.5,-3.5) -- node[above]{$f$} (4.5,-2.5);
\draw[->,dashed] (6.2,-3.5) -- node[right]{$\hat f$} (6.2,-2.5);
\end{tikzpicture}
\caption{An example of non-factorisation.}
\end{figure} 

However Viscardi's argument can be fixed. Let $(\mathcal{C}_\eta,F_{\eta})$ be a stable map to $\PP^r$, defined on the generic point of a DVR scheme $\dvr$; we may assume that $\mathcal C_{\eta}$ is smooth \cite[Section~3.2.1]{VISC}. 
As described in \cite[Theorem~3.6, Step~1]{VISC}, after applying nodal reduction we get a map $F\colon\mathcal C_{\dvr'}\to\PP^r_{\dvr'}$, for which we may suppose that $C:=\cC_0$ is nodal and reduced, and $f:=F_0$ is stable.

If $f$ is not constant on the minimal genus $1$ subcurve, then it is already $m$-stable and there is nothing to say. Otherwise let $E\subset C$ be the maximal genus $1$ subcurve where $f$ is constant and let $R_1\sqcup\ldots\sqcup R_m=\overline{C/E}$.

 By Proposition~\ref{prop:components}~(2) we know there is a non-trivial linear relation among the $\operatorname{d}\!f(T_{q_i}R_i)$'s.
  Consider a maximal one with all non-zero coefficients.  Possibly after relabelling, such a relation looks like:
 \begin{equation*}\label{eq:linrelation}
 \alpha_1\operatorname{d}\!f(T_{q_{1}}R_{1})+\ldots +\alpha_j \operatorname{d}\!f(T_{q_{j}}R_{j})=0.
 \end{equation*}
 In this case we blow-up $\mathcal{C}_\dvr$ in $q_{j+1},\ldots, q_m$. The induced map $\tilde{F_0}$ is constant on the exceptional divisors $G_{j+1},\ldots, G_m$ and we can complete the above linear relation to 
 \[\alpha_1\operatorname{d}\!\tilde f(T_{q_{1}}\tilde R_{1})+\ldots +\alpha_j \operatorname{d}\!\tilde f(T_{q_{j}}\tilde R_{j})+\beta_{j+1}\operatorname{d}\!\tilde f(T_{q_{j+1}}G_{j+1})+\ldots+\beta_m\operatorname{d}\!\tilde f(T_{q_{m}}G_m)=0\]
  with $\beta_i=1$ for all $i$. Now this \emph{sprouting} \cite[Section 2.3]{SMY2} ensures that the map descends to the corresponding elliptic $m$-fold singularity. Notice that, for every $m_1\leq m_2$, there is a morphism from an $m_2$-elliptic singularity to an $m_1$-elliptic point that is birational on the target and contracts $m_2-m_1$ branches of the source to the singular point. Proceed now with Step~2 of Viscardi's algorithm.
  
\end{remark}

The irreducible components of Viscardi's moduli space $\overline{\pazocal{M}}^{(m)}_{1,n}(\PP^r,d)$ are also well understood \cite[Thm.~5.9]{VISC}; indeed they have a similar description to the ones of Kontsevich's space. The main advantage of the $m$-stable compactification is that the number of components drops as $m$ increases.

In particular the space $\Mone{1}{\PP^r}{d}$ we will consider in the next sections has not got the boundary component $D^1$.
 
 \subsection{Genus $1$ singularities and smoothability}
 Inspired by Viscardi's alternate compactification, we give another description of smoothable maps in genus $1.$ We recall the following:
 \begin{dfn}
 Let $C$ be a reduced curve over $\k$ and $p\in C$ a singular point. We define the \emph{genus of the singularity} in $p$ to be the quantity: \[g(p)=\delta(p)-b(p)+1,\] where $\delta(p)=\dim_{\k}(\nu_*\OO_{\overline C}/\OO_C)\otimes k(p),$ $\nu\colon\overline{C}\to C$ is the normalisation of $C$ at $p$, and $b(p)$ is the number of branches of $C$ at $p.$
 \end{dfn}

 \begin{prop}\label{prop:smoothability2}
 Let $[f\colon C\to \PP^r]$ be a stable map from a genus $1$ curve $C=E {}_{\mathbf p}\!\sqcup_{\mathbf q}\bigsqcup_{i=1}^m R_i$, where $E$ is the maximal genus $1$ subcurve contracted by $f$, and the $R_i$ are rational tails on which $f$ has positive degree. 
  
  Then $f$ is smoothable if and only if it factors through a curve with a genus $1$ singularity $f\colon C\xrightarrow{\phi}\widehat C\xrightarrow{\hat f} \PP^r$, such that $\phi$ contracts $\operatorname{Exc}(f)$ and is an isomorphism outside it.
\end{prop}

\begin{lem}(Classification of genus $1$ singularities)
The Cohen-Macaulay (i.e. reduced) genus $1$ singularities with $m$ branches are obtained by gluing a genus $0$ singularity with $k$ branches together with a Smyth's $(m-k)$-elliptic fold.
Notice that $k$ may be $0$ (i.e. a point) or $1$ (i.e. a smooth rational curve).
\end{lem}
\begin{proof}
We extend the argument given by Smyth \cite[Appendix~A]{SMY1} to classify the Goreinstein genus $1$ singularities. We study the analytic germ of the singularity and we adopt Smyth's notation: $R$ denotes the completion of the local ring of $C$ at the singularity; 
$\widetilde{R}=\k[[t_1]]\oplus\ldots\oplus \k[[t_m]]$ its integral closure; $\mathfrak{m}_R$ the maximal ideal of $R$  and $\mathfrak{m}_{\widetilde{R}}$ that of $\widetilde{R}.$

Let us recall that to describe $R$ as a quotient polynomial ring, it is enough to find a $\k$-basis for $\mathfrak{m}_R/\mathfrak{m}^2_R=\langle e_1,\ldots, e_s\rangle_{\k}$ where the $e_i$ are some elements in $\widetilde{R}.$ Indeed, once given such a basis, it is a consequence of completeness that $R$ can be recognised as $\k[[x_1,\ldots, x_s]]/I$, where $I$ is the kernel of the ring homomorphism
\begin{align*}\label{eq:singularity}
\k[[x_1,\ldots, x_s]]&\rightarrow R\subset \k[[t_1]]\oplus\ldots\oplus \k[[t_m]]\\
\;\;x_i\;\; &\mapsto \;\; e_i\;\;
\end{align*}
Smyth observes that $\widetilde R/R$ is graded by:
\[ (\widetilde R/R)_i:=\mathfrak m_{\widetilde R}^i/(\mathfrak m_{\widetilde R}^i\cap R)+\mathfrak m_{\widetilde R}^{i+1}\]
Furthermore:
\begin{enumerate}
\item $m=\delta(p)=\sum_{i\geq 0}\dim_\k(\widetilde R/R)_i;$
\item $1=g(p)=\sum_{i\geq 1}\dim_\k(\widetilde R/R)_i;$
\item if $(\widetilde R/R)_i=(\widetilde R/R)_j=0$ then $(\widetilde R/R)_{i+j}=0$.
\end{enumerate}
From (2) and (3) it follows that $\dim_\k(\widetilde R/R)_1=1$ and $\dim_\k(\widetilde R/R)_i=0$ for $i\geq 2$. Then the exact sequence:
\[ 0\to \frac{\mathfrak m_{\widetilde R}^2}{\mathfrak m_{\widetilde R}^2\cap R}\to\frac{\mathfrak m_{\widetilde R}}{\mathfrak m_{\widetilde R}\cap R}\to (\widetilde R/R)_1\to 0\]
entails that:
\[\mathfrak{m}^2_{\widetilde{R}}\subseteq \mathfrak m_R,\qquad \mathfrak{m}_R/\mathfrak{m}^2_{\widetilde{R}}\subseteq \mathfrak{m}_{\widetilde{R}}/\mathfrak{m}^2_{\widetilde{R}}\;\;\text{is a codimension $1$ $\k$-subspace}.\]
Obviously $\mathfrak{m}^2_{R}\subseteq\mathfrak{m}^2_{\widetilde{R}}$. Hence $s$ is at least $m-1$. After relabelling we may assume $e_1,\ldots, e_{m-1}$ generate $\mathfrak{m}_R/\mathfrak{m}^2_{\widetilde{R}}$, and after Gaussian elimination they take the following form:
\begin{align*}
e_1&= (t_1, 0, \ldots, 0, a_1 t_m) \\
e_2&= (0, t_2, \ldots, 0, a_2 t_m) \\
 & \ldots \\
e_{m-1} &= (0, 0,\ldots, t_{m-1}, a_{m-1}t_m)
\end{align*}
with $a_{1}, \ldots, a_{m-1}\in \k.$ At this point Smyth restricts his attention to Gorenstein singularities and shows that under this assumption he can choose all the $a_i$ to be $1$. Furthermore $\mathfrak{m}^2_{\widetilde{R}}=\mathfrak{m}^2_{R}$ holds if $m\geq 3$, thus the above are generators for $\mathfrak{m}_R/\mathfrak{m}^2_R$ and the goal is reached. For $m=1$ (resp. $2$) he finds extra generators and shows they satisfy the equation of a cusp (resp. a tacnode).

Removing the Gorenstein restriction we have three possibilities for $m\geq 3$:
\begin{enumerate}
\item At least two of the $a_i$ are non-zero, say $i=1,2$: then $a_1a_2t_m^2=e_1\cdot e_2$ so $\mathfrak{m}^2_{\widetilde{R}}=\mathfrak{m}^2_{R}$ and we are done as above. It is easily seen that  $e_1,\ldots, e_{m-1}$ satisfy the equations of a $(k,m-k)$ singularity where $k$ is the number of $a_i$ that are zero.
\item Only one of the $a_i$ is non-zero, say $a_1=1$: then $\mathfrak{m}^2_{\widetilde{R}}=\mathfrak{m}^2_{R}+(t_m^2)$ and by adjoining $e_m=t_m^2$ to the $e_i$ we see that they generate $\mathfrak{m}_R/\mathfrak{m}^2_R$ and they satisfy the equation of a tacnode $(e_1^2-e_m)e_m=0$, and $e_ie_j=0$ for $i\neq j$ and $(i,j)\neq(1,m)$, so this is an $(m-2,2)$ singularity.
\item Finally all the $a_i$ are zero: then we have to add $t_m^2,\ t_m^3$ to generate $\mathfrak{m}_R/\mathfrak{m}^2_R$ and we find an $(m-1,1)$ singularity.
\end{enumerate}
\noindent Similarly for $m=2$ there are two possibilities: the tacnode (corresponding to $a_1\neq 0$) and the union of a cusp with a non-coplanar line (for $a_1=0$), with equations:
\[ \k[[x,y,z]]/(xz,yz,y^2-x^3).\]
\noindent For $m= 1$ the only possible singularity is the cusp: indeed it can be proven that $\mathfrak{m}^2_{\widetilde{R}}=\mathfrak{m}_{R}$. 
\end{proof}

\begin{remark}
 The genus and number of branches are no more sufficient to tell these singularities apart, neither is the embedding dimension. A numerical invariant that distinguishes them is:
\[ \dim_\k R/I_p \]
where 
\[I_p=\operatorname{Ann}_R(\widetilde R/R)=\left\{f\in\, R\;\rvert\; f\cdot\widetilde{R}\subseteq R\right\}\]
 is the conductor ideal at the singular point. 
 Using the explicit description of $R$ given in the lemma above, we can easily find generators of $I_p$ and check that for a $(k,m-k)$ singularity we have:
 \[ \dim_\k R/I_p=m-k.\]
\end{remark}

\begin{lem}\label{lem:sing_smoothability}
Every genus $1$ singularity is smoothable.
\end{lem}
\begin{proof}
 We explicitly construct a (semi)stable model with a contraction to the given singularity. Assume we start with a $(k,m-k)$-singularity $\widehat C$. Pick any $m$-pointed elliptic curve $(E,q_1,\ldots,q_m)$ and glue $E$ along the markings with $m$ copies of $\PP^1$ at their respective points $0$; call the rational tails $R_1,\ldots,R_m$. It is useful to consider the resulting curve $C$ as a point of $\oM_{1,2m}$ with markings on the rational tails given by $1$ and $\infty$. We now choose a smoothing $\cC_{\dvr}$ of $C$ over a DVR scheme $\dvr$, such that the total space $\cC_{\dvr}$ has an $A_1$-singularity at the nodes $q_1,\ldots,q_k$ and is regular everywhere else; furthermore extend the markings to get a horizontal Cartier divisor $\Sigma$. Resolving the singularities we obtain a fibered surface $\pi^{\rm{ss}}\colon\cC_{\dvr}^{\rm{ss}}\to\dvr$ such that in the central fiber the strict transforms $\widetilde R_1,\ldots,\widetilde R_k$ are at distance $1$ from the core $E$, while the other rational tails are adjacent to it. Call $F_1,\ldots,F_k$ the exceptional divisors of the resolution. As in \cite[Lemma 2.12]{SMY1} we may contract $E$, which is obviously balanced, by means of the line bundle:
 \[\mathcal L_1:=\omega_{\pi^{\rm{ss}}}(E+\Sigma),\]
 which results in $\overline{\cC}_{\dvr}$ with an $m$-elliptic singularity in the central fiber, $k$ of whose branches are the images of $F_1,\ldots,F_k$ and are thus unmarked. We may finally perform a second contraction by using the line bundle $\mathcal L_2:=\OO_{\overline{\cC}_{\dvr}}(\overline{\Sigma})$; let $\hC_{\dvr}$ denote the resulting family of curves, with singularity $\hat q$. Notice that $\mathcal L_2$ satisfies cohomology and base-change, hence we may simply check on the central fiber both its relative semi-ampleness and the behaviour of its sections. The sections will be constant along $\overline F_1,\ldots,\overline F_k$, so the linear relation among their derivatives at the Smyth's singularity will only imply the linear dependence of the tangent vectors of $\widehat R_{k+1},\ldots, \widehat R_m$ at $\hat q$; on the other hand sections of $\mathcal L_2$ along $\overline R_1,\ldots,\overline R_k$ are completely independent, and thus they embed these $\PP^1$. We deduce that the central fiber of $\hC_{\dvr}$ has a $(k,m-k)$-singularity.
\end{proof}

\begin{proof}\ref{prop:smoothability2}
 The argument that if $f$ is smoothable then it has to factor through a genus $1$ singularity was essentially given by Vakil in \cite[Lemma 5.9]{Vre}. We review it here in some detail.
 
 Pick a smoothing $F\colon\cC_{\dvr}\to \PP^r$ of $f=F_0$. After base-change we may assume that $\dvr$ is the spectrum of a complete DVR. The line bundle $\mathcal L:=F^*\OO_{\PP^r}(1)$ is trivial on every irreducible component of the central fiber contracted by $f$.  Since every connected component of $\operatorname{Exc}(\mathcal L)=\{D\subset C : \mathcal{L}\rvert_D\simeq \OO_D\}$ has arithmetic genus $0$ or $1$, we may use the argument in \cite[Lemma 2.12]{SMY1} to show that $\mathcal L$ is $\pi_{\dvr}$-semi-ample. We thus get a contraction $\phi\colon \cC_{\dvr}\to \overline{\cC}_{\dvr}$ and notice that $F$ factors through $\phi$ by the construction of $\overline{\cC}_{\dvr}=\underline{\operatorname{Proj}}_{\dvr}\left(\bigoplus_{n\geq 0}\pi_{\dvr,*}\mathcal L^{\otimes n}\right)$.
 Being a smoothing of a reduced curve, $\cC_{\dvr}$ is normal, thus $\phi$ factors through the normalisation $\nu\colon \hC_{\dvr}\to \overline{\cC}_{\dvr}$, which is a finite map by \cite[\S 8.2]{Liusbook}. It is now clear that $\hC_{\dvr}\to\dvr$ is a flat family of genus $1$ curves with reduced fibers, and the map $\overline{F}\circ\nu$ has positive degree on all the components of $\hC_0$.

 Viceversa let us suppose that we have a factorisation:
 \bcd
 C\ar[rr,"f"]\ar[dr,"\phi"] & & \PP^r \\
 & \widehat{C}\ar[ur,"\hat f"] &
 \ecd
 such that $\phi$ contracts $\operatorname{Exc}(f)$ and is an isomorphism everywhere else. We shall below make the point that $f$ is smoothable as soon as it is not constant on the core (compare also with \cite[Theorem 4.5.1]{RSPW}), so in this case we have no use for $\phi$. Otherwise $\phi$ contracts at least the core, and $\widehat C$ has a genus $1$ singularity at a certain point $q$. We first show that:
 \begin{claim}
 Under these assumptions, we can produce smoothings $\cC_{\dvr}$ of $C$ and $\hC_{\dvr}$ of $\widehat{C}$ with a contraction map $\phi_{\dvr}$ extending the given $\phi$.
 \end{claim}
  A statement of this kind is in general far from true, see Remark~\ref{rem:compatible_smoothing} below.
  The proof of this claim follows the lines of the two-step contraction described in the proof of Lemma~\ref{lem:sing_smoothability}. It will be useful to treat $C$ as a stable curve marked with $f^{-1}(Q)$ for a generic quadric $Q\subseteq \PP^r$.
  \begin{proof} Assume that $\widehat C$ has a $(k,m-k)$-singularity at $q$; denote by $E:=\phi^{-1}(q)\subset C$, which is the maximal genus $1$ unmarked (i.e. contracted by $f$) connected subcurve of $C$. By assumption the components of $C$ adjacent to $E$ are marked (i.e. not contracted by $f$), and they can be partitioned into $R_1,\ldots,R_k$ mapping to the genus $0$ part of the $(k,m-k)$-singularity, and $R_{k+1},\ldots,R_m$ parametrising the branches of the $(m-k)$-elliptic fold. Let $l_i$ denote the distance of $R_i$ from the core $Z$, which can be determined from the marked dual graph. Let $l=\max\{l_{k+1},\ldots,l_m\}$ and $L=\max\{l_1,\ldots,l_m\}$. Consider a smoothing $\cC_{\dvr}$ of $C$ with an $A_{l-l_i}$ singularity at $R_i\cap E$ for $i=k+1,\ldots,m$, and an $A_{L-l_i+1}$-singularity at $R_i\cap E$ for $i=1,\ldots,k$. Extend the markings $f^{-1}(Q)$ to a horizontal Cartier divisor $\Sigma$. Let $\cC_{\dvr}^{\rm{ss}}$ be the semistable model with regular total space; notice that in the central fiber the strict transforms of $R_{k+1},\ldots,R_m$ are at distance $l$ from the core, while all other marked components are further apart. Let $E^{\rm{bal}}$ be the maximal unmarked balanced subcurve inside $\cC_0^{\rm{ss}}$, including all the components at distance less than $l$ from $Z$. We may now use Smyth's line bundle:
 \[\mathcal L_1=\omega_{\cC^{\rm{ss}}_{\dvr}/\dvr}\left( \sum_{F\subseteq E^{\rm{bal}}}(l+1-l(F,Z))F\right)\otimes \OO_{\cC^{\rm{ss}}_{\dvr}}(\Sigma),\]
 to perform the contraction of $E^{\rm{bal}}$, obtaining a family $\overline{\cC}_{\dvr}\to\dvr$ whose central fiber has an $m$-elliptic fold point. Notice that the contraction is otherwise an isomorphism, due to the stability condition.
 As in Lemma~\ref{lem:sing_smoothability} we may then perform a second contraction by using the line bundle $\mathcal L_2:=\OO_{\overline{\cC}_{\dvr}}(\overline{\Sigma})$; we denote by $\hC_{\dvr}$ the resulting family of curves. It may be seen that the central fiber $\hC_0$ is isomorphic to $\widehat C$, and so is the resulting contraction up to an automorphism of $C$. \end{proof}

Let $\hL$ be an extension of $\widehat{L}:=\hat{f}^*\OO_{\PP^r}(1)$ on $\hC$, which exists because deforming line bundles on curves is unobstructed.
In order to extend $\hat{f}$ to $\widehat{F}\colon \hC_{\dvr}\to\PP^r$ all we have to show is that the $r+1$ sections $\hat{s}_0,\dots, \hat{s}_{r+1}$ representing the map $\hat{f}$ extend to sections of $\hL.$
Thus it is enough to verify that $H^1(\widehat{C},\widehat{L})=0$ \cite{Wang}. Once this is done we are going to obtain the smoothing of the original map by precomposing with the contraction $F\colon\cC_{\dvr}\xrightarrow{\phi}\hC_{\dvr}\xrightarrow{\widehat F}\PP^r$.

Since all the curves we are considering are Cohen-Maculay, they have a dualising sheaf $\omega_{\widehat{C}}$ such that for any line bundle $\widehat{L}$:
\[H^1(\widehat{C},\widehat{L})=(H^0(\widehat{C},\widehat{L}^{-1}\otimes\omega_{\widehat{C}}))^{\vee}.\]
It is known \cite[IV, \S~3]{serre} that  $\omega_{\widehat{C}}$ can be described as the subsheaf of $\nu_* \omega_{\overline{C}}\otimes K(\overline{C})$ (where $\nu\colon\bar C\to\widehat{C}$ is the normalisation) satisfying: for every open $U\subseteq\widehat{C}$ a rational $1$-form $\omega\in \nu_*\left( \omega_{\overline{C}}\otimes K(\overline{C})\right)(U)$ is a section of $\omega_{\widehat{C}}(U)$ if and only if:
\[\sum_{q\in \nu^{-1}(0)}\rm{Res}(\nu^*(f)\omega) =0 \quad \quad \forall f\in\OO_{\widehat C}(U)\]

From this description it is patent that the pullback of $\omega_{\widehat{C}}$ to the normalisation restricts to a line bundle of degree $-1$ or $0$ on every branch, according to it being one of the $k$ or $m-k$ components. Hence it can be seen from the normalisation exact sequence and Serre duality that $H^1(\widehat{C},\widehat{L})=0$ as soon as $\hat L$ has positive degree on one of the branches.
 \end{proof}

 \begin{remark}\label{rem:compatible_smoothing}
In general it is not possible to find compatible smoothings as above: pick any genus $2$ curve with a rational tail and map it to the ramphoid cusp, a planar singularity with local equation $y^2-x^5=0$. Then a compatible smoothing can be found only if the rational tail is attached to a Weierstrass point of the genus $2$ curve, as can be seen by computing the semi-stable model of the ramphoid cusp. This example was kindly suggested to us by Prof. D.I. Smyth.
\end{remark}

 \begin{remark}
A way around the construction of a compatible smoothing, would be to appeal to \cite[Lemma 2.9]{SMYtowards} and prove that the smoothability of $f$ only depends on its restriction to the $k$ rational tails, namely it is independent of the $k$-pointed elliptic curve that $f$ contracts.
\end{remark}

\section{The moduli space of $1$-stable maps with $p$-fields}\label{section:p-fields}
We adapt Chang-Li's theory of $p$-fields to cuspidal maps. This is a word-by-word repetition of the arguments in \cite{CLpfields} once noticed that they carry over to families of at worst cuspidal curves. We provide the non-expert reader with a r\'{e}sum\'{e} of some of the key ideas contained in \cite{CLpfields}; this section can otherwise be skipped.
\subsection{Moduli of sections}
\begin{dfn}
 Let $B$ be an algebraic stack and let $\pi\colon \cC\to B$ be a flat proper morphism of finite presentation, which is representable by algebraic spaces, and whose geometric fibers are reduced l.c.i. curves. Let $\mathcal Z$ be an algebraic stack,  representable, quasi-projective, and smooth over $\cC$. The \emph{cone of sections} of $\mathcal Z$ over $\cC$ is a $B$-stack $\mathfrak S$ defined by:
 \[\mathfrak{S}(S\to B)=\left\{\text{sections of}\ \mathcal Z_S\to\cC_S\right\};\]

 The groupoid $\mathfrak{S}$ is an algebraic stack representable and quasi-projective over $B$ \cite[Proposition~2.3]{CL}.
 
\end{dfn}
 To fix the notation let us draw the following diagram:
 \bcd
& \mathcal Z\ar[d] \\
\cC_{\mathfrak S}\ar[r]\ar[ur,bend left,"\mathfrak e"]\ar[d,"\pi_{\mathfrak S}"] & \cC\ar[d,"\pi"] \\
\mathfrak S \ar[r] & B
\ecd
where $\mathfrak e$ denotes the universal section. 
Let us see some examples of the above construction we will be interested in. 

\subsubsection*{Direct image cones}
When $\mathcal Z=\rm{Vb}(\mathcal L)$ for a line bundle $\mathcal L$ on $\cC$, the algebraic stack  $\mathfrak{S}\to B$ representing sections of $\mathcal L$ can be constructed as:   
\[C(\pi_*\mathcal L):=\Spec_B\operatorname{Sym}^{\bullet}(\R^1\pi_*\mathcal L^{\vee}\otimes\omega_{\pi}).\]
 This is essentially because $ \R^1\pi_*\mathcal L^{\vee}\otimes\omega_{\pi}$ commutes with pullbacks, and it has the desired modular interpretation by Serre duality. 

\subsubsection*{Moduli of stable maps and $1$-stable maps} 
Recall that $\mathfrak{P}$ denotes the universal Picard stack.
Let $X\subseteq \PP^r$ be a projective variety and let $\mathcal Z_X$ be defined by the following cartesian diagram over $\cC_{\mathfrak{P}}$:
\bcd
\mathcal Z_X\ar[r,"/\Gm"]\ar[d,hook]\ar[dr,phantom,"\Box"] & X\times \cC_{\mathfrak{P}}\ar[d,hook] \\
\mathcal L^{\oplus r+1}\setminus 0_{\cC}\ar[r,"/\Gm"] & \PP^r\times\cC_{\mathfrak{P}}
\ecd
Then $\oM_1(X,\beta)$ is the open substack of the moduli space $\mathfrak{S}_X$ of sections of $\mathcal Z_X\to\cC_{\pP}$ defined by the stability condition.  This is the point of view we have already taken for $X=\PP^n$ in Proposition \ref{prop:smoothability} to describe the local equations of the moduli space of maps.

Analogously the moduli space of $1$-stable maps $\oM_1^{(1)}(X,\beta)$ can be thought of as an open inside the moduli space of sections of 
\[\widehat{\mathcal Z}_X=(\operatorname{Vb}(\widehat{\mathcal{L}}^{\oplus r+1}_{\hP})\setminus 0_{\widehat{\cC}})\times_{\PP^r} X\]
over $\widehat{\cC}_{\hP}.$

\subsubsection*{Moduli of $1$-stable maps with $p$-fields}
 In this section we denote by $\hrM$ the moduli space of $1$-stable maps $\oM_1^{(1)}(\PP^4,d)$ and by $\widehat{\mathcal P}_{\hrM}=\hL_{\hrM}^{\otimes -5}\otimes\omega_{\hat\pi_{\hrM}}.$ 
 
The moduli space of $p$-fields is defined as the cone of sections of the line bundle $\widehat{\mathcal P}_{\hrM}$ over $\hC_{\hrM}$:

\begin{dfn}
 The moduli space of $1$-stable maps with $p$-fields: \[\Mone{1}{\PP^4}{d}^p:=C(\hat\pi_{\hrM,*}(\widehat{\mathcal P}_{\hrM}))\] parametrises $1$-stable maps:
 \bcd
 \hC_S\ar[r,"\hat f_S"]\ar[d,"\hat\pi_S"] & \PP^4 \\
 S &
 \ecd
 with a $p$-field $\hat\psi\in H^0(\hC_S,\hat f_{S}^*\OO_{\PP^4}(-5)\otimes\omega_{\hat\pi_{S}})$.
\end{dfn}

We use the abbreviation $\widehat{\pazocal P}:=\Mone{1}{\PP^4}{d}^p$. Employing the above description of the moduli space of $1$-stable maps,  $\widehat{\pazocal P}$ can also be thought of as an open in the moduli space of sections of the vector bundle $\rm{Vb}(\hL_{\hP}^{\oplus 5}\oplus \widehat{\mathcal P}_{\hP}).$

\subsection{Obstruction theories}
With the above, the morphism $\mathfrak S\to B$ admits a relative dual perfect obstruction theory:
\[ \phi_{\mathfrak S/B}\colon \mathbb T_{\mathfrak S/B}\to \mathbb E_{\mathfrak S/B}:=\R^\bullet\pi_{\mathfrak S,*}\mathfrak e^* T_{\mathcal Z/\cC}\]
For the proof see \cite[Proposition 2.5]{CLpfields} and notice that it relies on general properties of obstruction theories and the cotangent complex, and the fact that $\mathcal Z\to\cC$ is smooth,
but never on the specification that $\cC\to B$ is a family of \emph{nodal} curves.

Let us review the examples above:
\begin{itemize}
 \item for the direct image cone of a line bundle $\mathcal L$ the dual obstruction theory is $\mathbb E_{C(\pi_*\mathcal L)/B}=\R^\bullet\pi_*\mathcal L$;
 \item the moduli space $\hrM=\oM_1^{(1)}(\PP^4,d)$ has a dual obstruction theory relative to $\hP$ given by $\mathbb{E}_{\hrM/\hP}= \R^\bullet\hat\pi_{\hrM,*}(\bigoplus_0^4 \hL_{\hrM})$;
 \item in the case of $1$-stable maps with $p$-fields $\widehat{\pazocal P}=\oM_1^{(1)}(\PP^4,d)^p$, we get $\mathbb E_{\widehat{\pazocal P}/\hP}=\R^\bullet\hat\pi_{\widehat{\pazocal P},*}(\hL_{\widehat{\pazocal P}}^{\oplus 5}\oplus\widehat{\mathcal P}_{\widehat{\pazocal P}}).$
\end{itemize}
We review the compatibility of various obstruction theories for the moduli spaces mentioned above. 
\begin{lem}
 $\hP$ is a smooth Artin stack of dimension $0$. Furthermore there is a compatible triple of dual perfect obstruction theories:
 \bcd
 \hat\rho ^*\mathbb T_{\hP/\hM}[-1]\ar[r]\ar[d,"\wr"] &  \mathbb T_{\hrM/\hP} \ar[r]\ar[d] &  \mathbb T_{\hrM/\hM} \ar[r,"{[}1{]}"]\ar[d] & {} \\
 \R^\bullet\hat\pi_{\hrM,*}(\OO_{\hC})\ar[r] & \R^\bullet\hat\pi_{\hrM,*}(\bigoplus_0^4 \hL_{\hrM})\ar[r] & \R^\bullet\hat\pi_{\hrM,*}(f_{\hrM}^* T_{\PP^4})\ar[r,"{[}1{]}"] & {}
 \ecd
 implying that $ \mathbb E_{\hrM/\hP}=\R^\bullet\hat\pi_{\hrM,*}(\bigoplus_0^4 \mathcal L_{\hrM})$ gives the standard Behrend-Fantechi-Viscardi virtual class on $\hrM$.
\end{lem}
\begin{proof}
 The first statement follows from deformation theory: the projection $\hat\rho\colon\hP\to\hM$ is unobstructed of relative dimension $0$ and $\hM$ is a smooth Artin stack of dimension $0$, since both nodal and cuspidal singularities are l.c.i., so obstructions to their deformations are contained in $\operatorname{Ext}^2_{\OO_{\widehat{C}}}(\Omega_{\widehat{C}},\OO_{\widehat{C}})=0$.
 
 The fact that $\mathbb T_{\hrM/\hM}\to \mathbb E_{\hrM/\hM}:=\R^\bullet\hat\pi_{\hrM,*}(f_{\hrM}^* T_{\PP^4})$ is a perfect obstruction theory when $\hC\to\hrM$ is a family of Gorenstein curves is proved in \cite[Proposition 6.3]{BF}.
 
 The lower row in the above diagram is induced by the Euler sequence of $\PP^4$. The middle column comes from identifying the space of stable maps as an open substack of the cone of sections (see above) of $\rm{Vb}(\bigoplus_0^4\hL)$ over $\hP$. The existence of such a commutative diagram is \cite[Lemma 2.8]{CLpfields}.
 The final claim follows from functoriality of virtual pullbacks \cite{Manolache-pullback}.
\end{proof}

\begin{lem}
 There is a compatible triple of dual perfect obstruction theories \[(\R^\bullet\hat\pi_{\widehat{\pazocal P},*}(\widehat{\mathcal P}_{\widehat{\pazocal P}}),\R^\bullet\hat\pi_{\widehat{\pazocal P},*}(\hL_{\widehat{\pazocal P}}^{\oplus 5}\oplus\widehat{\mathcal P}_{\widehat{\pazocal P}}),\R^\bullet\hat\pi_{\widehat{\pazocal P},*}(\hL_{\widehat{\pazocal P}}^{\oplus 5}))\] for the triangle:
 \bcd
 \widehat{\pazocal P}\ar[rr,"\hat\rho"]\ar[dr] & & \hrM\ar[dl] \\
 & \hP &
 \ecd
\end{lem}

Notice that the virtual rank of $\mathbb E_{\widehat{\pazocal P}/\hP}:=\R^\bullet\hat\pi_{\widehat{\pazocal P},*}(\hL_{\widehat{\pazocal P}}^{\oplus 5}\oplus\widehat{\mathcal P}_{\widehat{\pazocal P}})$ is $0$, hence it endows the moduli space of $1$-stable maps with $p$-fields with a cycle class of dimension $0$. However $\widehat{\pazocal P}$ is not proper. In the next section we describe Chang-Li's cosection (depending on the choice of $\w\in\k[x_0,\ldots,x_4]_5$) of the obstruction bundle, and show that the corresponding localised cycle is supported on the proper subtack $\Mone{1}{X}{d},$ where $X=V(\w)$ is the quintic threefold, that is the degeneracy locus of the cosection.

\subsection{Cosection localisation and virtual pullback}\label{sec:cosection}
Recall Kiem-Li's machinery of cosection localised virtual classes \cite[Theorem 1.1]{KLcosection}:

\begin{teo*}[Localisation by cosection]\label{thm:KLlocalisation}
Let $Y\to S$ be a morphism of DM type between algebraic stacks, with $Y$ Deligne-Mumford and $S$ smooth, endowed with a perfect
obstruction theory. Suppose the obstruction sheaf $\rm{Ob}_Y$ admits a
surjective homomorphism $\sigma:\rm{Ob}_{Y|_U } \to \OO_U$
over an open $U\subseteq Y$.

Let $\iota\colon Y(\sigma):=Y\setminus U\hookrightarrow Y$,
then $(Y,\sigma)$ has a localised virtual cycle:
$$[Y]^{\rm{vir}}_{\rm{loc}}\in A_* Y(\sigma).$$
This cycle enjoys the usual properties of the virtual cycles; it relates to
the usual virtual cycle $[Y]^{\rm{vir}}$ via 
$[Y]^{\rm{vir}} = \iota_*[Y]^{\rm{vir}}_{\rm{loc}} \in A_*Y  .$
\end{teo*}
We now review a slight generalisation of this, that combines it with Manolache's virtual pullback construction \cite{Manolache-pullback}; this can also be found in \cite[\S~5]{CLpfields}.

Let $Y\to S$ be as in the hypotheses of the theorem, or more generally such that there is a triple of compatible obstruction theories for a triangle $Y\to S\to T$ with $T$ smooth. Assume for simplicity that the obstruction theory $\mathbb E_{Y/S}$ is a vector bundle $E$  admitting a cosection:
\[E\xrightarrow{\sigma}\OO_Y\]
surjective on an open $U\subseteq Y$; denote by $Y(\sigma)$ the complement $Y\setminus U.$ Recall the following notation from \cite{KLcosection}:
\[G=\operatorname{Ker}\left(E\rvert_U\to\OO_Y\rvert_U\right),\qquad E(\sigma)=E\rvert_{Y(\sigma)}\cup G.\]
 Kiem and Li define a localised Gysin map:
 \[s^!_{\sigma,\rm{loc}}\colon A_*(E(\sigma))\to A_*(Y(\sigma)),\]
 which we explain in what follows.
 Let $[B]\in Z_*(E(\sigma))$ be a cycle represented by a closed integral substack. If $B\subset  E\rvert_{Y(\sigma)}$ then
  they use the standard Gysin morphism: \[s^!_{\sigma,\rm{loc}}[B]:=s^!_{E\rvert_{Y(\sigma)}}[B]\in A_*(Y(\sigma)).\]
 Suppose instead that $B$ is not contained in $E\rvert_{Y(\sigma)}.$ Then we may choose a variety with a \emph{regularising morphism} $\widetilde{Y}\xrightarrow{\nu} Y$ such that:
 \begin{itemize}\label{cond:deflocal}
 \item $\nu$ is proper and $\nu(\widetilde{Y})\cap U\neq\emptyset$;
 \item the pullback along $\nu$ of the cosection $\nu^*\sigma=:\tilde{\sigma}$ extends to a surjective morphism
 \[\nu^*E\xrightarrow{\tilde{\sigma}}\OO_{\widetilde{Y}}(D)\to 0\]
 for a Cartier divisor $D\subseteq \widetilde{Y}.$
 \item there is a closed integral $\widetilde{B}\subset\widetilde{G}=\operatorname{Ker}(\tilde{\sigma})$ such that $\tilde{\nu}_*[\widetilde{B}]=k[B]$ for some $k\in\mathbb N.$
 \end{itemize}
 We denote by  $\nu(\sigma)\colon D\to Y(\sigma)$ the restriction of $\nu$ to the divisor, and by $\tilde{\nu}\colon\widetilde{G}\to E(\sigma)$ the restriction of the natural map $\nu^*E\to E.$
 The localised Gysin pullback is then defined as:
 \[s^!_{\sigma,\rm{loc}}[B]:=\frac{1}{k} \nu(\sigma)_*\left([D]\cdot s^!_{\widetilde{G}}[\widetilde{B}]\right)\,\in\,A_*(Y(\sigma)).\]
 In \cite[\S~2]{KLcosection} they prove that: such $(\widetilde{Y},\nu,\widetilde B)$ always exist; the cycle $s^!_{\sigma,\rm{loc}}[B]$ is independent of the above choices; the construction preserves rational equivalences and thus defines the desired morphism on Chow groups.
 \begin{remark}
 If we consider $[Y]$ as a cycle in $E(\sigma)$, thought of as the zero section of $E$, notice that a natural choice for $\widetilde{Y}$ is $\operatorname{Bl}_{Y(\sigma)}Y.$
 \end{remark}
  Kiem and Li then proceed to extend the cosection-localised Gysin pullback to vector bundle stacks $h^1/h^0(\mathbb E)$ (this is relatively easy in the presence of a global resolution, which is the case in the situation at hand); they show that the intrinsic normal cone $\mathfrak C_{Y/S}$ is contained in the closed substack $h^1/h^0(\mathbb{E})(\sigma),$ at least when the cosection of the relative obstruction bundle lifts to a cosection of the absolute obstruction bundle \cite[Corollary~4.5]{KLcosection}. 
\begin{dfn}
Consider a cartesian diagram of stacks:
\bcd
Y'\ar[r,"\varphi'"]\ar[d,"\rho'"]\ar[dr,phantom,"\Box"] & Y\ar[d,"\rho"]\\
S'\ar[r,"\varphi"] & S
\ecd
with $\rho$ as above. Assume furthermore that we have a cosection 
\[\rm{Ob}_{Y/S}\to \OO_Y\]
that lifts to a cosection of the absolute obstruction sheaf $\rm{Ob}_Y$.
Then we can define a localised virtual pullback operation: \[(\rho')^!_{\mathbb E_{Y'/S'},\sigma'}\colon A_*(S')\to A_{*-\operatorname{rk}(\mathbb E)}(Y'(\sigma'))\] 
where $\mathbb E_{Y'/S'}=(\varphi')^*\mathbb E_{Y/S},\;\sigma'=(\varphi')^*(\sigma)$ and $Y'(\sigma')=Y(\sigma)\times_{Y} Y'$.
\end{dfn}
Replace $S'$ by a primitive cycle in it, i.e. assume that $S'$ is irreducible and reduced. Recall that pulling back $\mathbb E_{Y/S}$ along $\varphi'$ endows $\rho'$ with a relative perfect obstruction theory.  Furthermore we can pullback the cosection:
\[\rm{Ob}_{Y'/S'}\cong \varphi^* \rm{Ob}_{Y/S}\xrightarrow{\sigma'}\OO_{Y'}.\]
 The degeneracy locus of $\sigma'$ is $Y'(\sigma'),$ which in turn implies that:
\[h^1/h^0(\mathbb E_{Y'/S'})(\sigma')=(\varphi')^*h^1/h^0(\mathbb E_{Y/S})(\sigma).\]

By \cite[Corollary 4.5]{KLcosection} and from the cartesian diagram above,
\[\mathfrak{C}_{Y'/S'}\subseteq (\varphi')^* \mathfrak{C}_{Y/S}\subseteq \mathbb E_{Y'/S'}(\sigma')\]
thus we can define the localised virtual pullback of $[S']$ to be:
\[(\rho')^!_{\mathbb E_{Y'/S'},\sigma'}[S']=s^!_{\sigma',\rm{loc}}[\mathfrak{C}_{Y'/S'}].\]

\subsection{A cosection for $p$-fields}

We are going to construct the cosection paralleling \cite[\S\S3.2-3.4]{CLpfields}. There is a morphism of vector bundles on $\hP$ induced by iterated tensoring of line bundles:
\[ h_1\colon\rm{Vb}(\hL_{\hP}^{\oplus 5}\oplus\widehat{\mathcal P}_{\hP})\to \rm{Vb}(\omega_{\hat\pi_{\hP}}),\quad h_1(x,p)=p\w(x_0,\ldots,x_4)\]
By differentiating it and pulling it back along the universal evaluation
\bcd
& \rm{Vb}(\hL_{\hP}^{\oplus 5})\setminus\{0\}\oplus\rm{Vb}(\widehat{\mathcal P}_{\hP})\ar[d] \\
\hC_{\widehat{\pazocal P}}\ar[ru,bend left,start anchor=north, end anchor=west,"\mathfrak e"]\ar[d,"\hat\pi_{\widehat{\pazocal P}}"]\ar[r] & \hC_{\hP}\ar[d,"\hat\pi_{\hP}"] \\
\widehat{\pazocal P} \ar[r] & \rm \hP
\ecd
 we obtain a cosection of the relative obstruction sheaf
\begin{equation}\label{eqn:cosection}
 \begin{split}
 \sigma_1\colon\rm{Ob}_{\widehat{\pazocal P}/\hP}=\R^1\hat\pi_{\widehat{\pazocal P},*}(\hL_{\widehat{\pazocal P}}^{\oplus 5}\oplus\widehat{\mathcal P}_{\widehat{\pazocal P}})\to \R^1\hat\pi_{\widehat{\pazocal P},*}(\omega_{\hat\pi_{\widehat{\pazocal P}}})\simeq \OO_{\widehat{\pazocal P}} \\
 \sigma_{1|(u,\psi)}(\mathring{x},\mathring{p})=\mathring{p}\w(u)+\psi\sum_{i=0}^4\partial_i\w(u)\mathring{x}_i
\end{split}
\end{equation}

The degeneracy locus of this cosection is $\Mone{1}{X}{d}$: by Serre duality if $\w(u)\neq 0$ then we can find a $p$ such that the cosection does not vanish; similarly we can do if $\psi\sum_{i=0}^4\partial_i\w(u)\neq 0$. But $\w(u)=0$ and $\partial_i\w(u)=0$ never happen simultaneously by smoothness of $X$, so it has to be $\w(u)=\psi=0$. 

Moreover Chang and Li prove that $\sigma_1$ lifts to a cosection of the absolute obstruction bundle $\rm{Ob}_{\widehat{\pazocal P}}\to\OO_{\widehat{\pazocal P}};$ it has the same degeneracy locus because $\rm{Ob}_{\widehat{\pazocal P}/\hP}\to \rm{Ob}_{\widehat{\pazocal P}}$ is surjective. This is a sufficient condition for the relative intrinsic normal cone to be contained in the closed substack of the obstruction bundle determined by the cosection, see \S~\ref{sec:cosecvirpull}.

We may thus endow $\Mone{1}{\PP^4}{d}^p$ with a localised virtual cycle:
\[\virloc{\widehat{\pazocal P}}=0^{!}_{\sigma_1,\rm{loc}}[\mathfrak{C}_{\widehat{\pazocal P}/\hP}]\;\in\; A_0\left(\Mone{1}{X}{d}\right).\]

We want to show that it gives the same numerical invariants as the cuspidal Gromov-Witten theory of $X$, up to a sign:
\begin{thm}\label{thm:p-fields-quintic}
 \[\deg[\Mone{1}{\PP^4}{d}^p]^{\rm{vir}}_{\rm{loc}}= (-1)^{5d}\deg[\Mone{1}{X}{d}]^{\rm{vir}}\]
\end{thm}

\subsection{From $p$-fields to the quintic threefold}
This is achieved in two steps: we first compare the invariants of $\Mone{1}{\PP^4}{d}^p$ with the ones of $\Mone{1}{N_{X/\PP^4}}{d}^p$, where $N_{X/\PP^4}$ is the normal bundle of the quintic in $\PP^4.$

Second we compare the latter invariants with $\deg[\Mone{1}{X}{d}]^{\rm{vir}}.$

\subsubsection*{Deformation to the normal cone}
This is attained in \cite[\S\S4-5]{CLpfields} by a family version of the $p$-fields construction applied to the deformation to the normal cone of $X\subseteq \PP^4$; let us denote the latter by $V\to\Aaff^1_t$, so that $V_{t\neq 0}=\PP^4$ and $V_0=N_{X/\PP^4}.$

\begin{lem}
 The deformation to the normal cone $V$ is cut inside $\rm{Vb}(\OO_{\PP^4}(5))\times\Aaff^1_t$ with basis coordinates $[x_0:\ldots:x_4]$ and fiber coordinate $y$ by the equation $\w(x)-ty=0$. If $C(V)$ denotes the affine cone over $V$, then its tangent bundle is determined by the following exact sequences:
 \begin{align}
  0\to T_{C(V)/\Aaff^1_t}\to\OO^{\oplus 5}_{C(V)}\oplus\OO_{C(V)}\xrightarrow{\sum_i\partial_i\w(x)\mathring{x_i}-t\mathring{y}}\OO_{C(V)}\to 0 \label{eq:reltangentcone}\\
    0\to T_{C(V)}\to\OO^{\oplus 5}_{C(V)}\oplus\OO_{C(V)}\oplus\OO_{C(V)}\xrightarrow{\sum_i\partial_i\w(x)\mathring{x_i}-t\mathring{y}-y\mathring{t}}\OO_{C(V)}\to 0\label{eq:tangentcone}
  \end{align}
\end{lem}
This allows a description of the moduli space of maps to $V$ as the cone of sections of a certain smooth object $\mathcal Z'$ over $\hC_{\hP\times\Aaff^1}$:

\bcd
& \mathcal Z'\ar[d]\ar[r]\ar[dr,phantom,"\Box"] & V\ar[d] \\
& \rm{Vb}(\mathcal L_{\hP}^{\oplus 5})\setminus\{0\}\oplus\rm{Vb}(\mathcal L^{\otimes 5}_{\hP})\ar[d] \ar[r] & \rm{Vb}(\OO_{\PP^4}(5))\times\Aaff^1_t\\
\hC_{\oM_1^{(1)}(V)}\ar[uur,bend left,"\mathfrak e"]\ar[d,"\hat{\pi}_{\oM_1^{(1)}(V)}"]\ar[r] & \hC_{\hP\times\Aaff^1_t}\ar[d,"\hat\pi_{\hP\times\Aaff^1_t}"] & \\
\Mone{1}{V}{(d,0)} \ar[r] & \rm \hP\times\Aaff^1_t &
\ecd

Similarly $\widehat{\pazocal V}:=\Mone{1}{V}{(d,0)}^p$ can be defined as the cone of sections of $\mathcal Z:=\mathcal Z'\oplus\rm{Vb}(\widehat{\mathcal P}_{\hP})$. The general theory explained above provides an obstruction theory for $\widehat{\pazocal V}\to\hP\times\Aaff^1_t$ \cite[Proposition 4.2]{CLpfields}:

\begin{lem}
 A dual perfect obstruction theory is given by:
 \[\phi_{\widehat{\pazocal V}/\hP\times\Aaff^1_t}\colon\mathbb T_{\widehat{\pazocal V}/\hP\times\Aaff^1_t}\to\mathbb E_{\widehat{\pazocal V}/\hP\times\Aaff^1_t}:=\R^\bullet\hat\pi_{\widehat{\pazocal V}}(f_{\widehat{\pazocal V}}^*\mathcal H\oplus\widehat{\mathcal P}_{\widehat{\pazocal V}})\]
 where $f_{\widehat{\pazocal V}}\colon \hC_{\widehat{\pazocal V}}\to V$ is the universal map and $\mathcal H$ is the vector bundle on $V$ defined by the exact sequence, see \eqref{eq:reltangentcone}:
 \[0\to\mathcal H\to\pr_{\PP^4}^*\left(\OO_{\PP^4}(1)^{\oplus 5}\oplus \OO_{\PP^4}(5)\right)\xrightarrow{\sum_i\partial_i\w(x)\mathring{x}_i-t\mathring{y}} \pr_{\PP^4}^*\OO_{\PP^4}(5)\to 0.\]
The restriction of $\phi_{\widehat{\pazocal V}/\hP\times\Aaff^1_t}$ to a fiber
 \[
 \widehat{\pazocal V}_t=\left\{\begin{array}{lr} \widehat{\pazocal P} & t\neq 0 \\ \Mone{1}{N_{X/\PP^4}}{d}^p & t=0 \end{array}\right.
\]
gives the standard obstruction theory of $\widehat{\pazocal V}_t\to\hP$.
\end{lem}

We would like to conclude that the restriction of the virtual cycle to the fibers is the standard virtual cycle on the fiber. The techniques of functoriality in intersection theory teach us that we should look for a triple of compatible obstruction theories for the triangle:
\bcd
\widehat{\pazocal V}_t\ar[rr,hook,"\iota_t"]\ar[dr] & & \widehat{\pazocal V} \ar[dl] \\
& \hP &
\ecd
The cone of sections interpretation provides us with an obstruction theory relative to $\widehat{\pazocal V}\to \hP$ given by:
\[
 \mathbb E'_{\widehat{\pazocal V}/\hP}:=\R^\bullet\hat\pi_{\widehat{\pazocal V}}(f_{\widehat{\pazocal V}}^*\mathcal K\oplus\widehat{\mathcal P}_{\widehat{\pazocal V}})
\]
where $\mathcal K$ is determined by the following exact sequence on $V$, see \eqref{eq:tangentcone}:
\[
 0\to\mathcal K\to\pr_{\PP^4}^*\left(\OO_{\PP^4}(1)^{\oplus 5}\oplus \OO_{\PP^4}(5)\oplus\OO_{\PP^4}\right)\xrightarrow{\sum_i\partial_i\w(x)\mathring{x}_i-t\mathring{y}-y\mathring{t}} \pr_{\PP^4}^*\OO_{\PP^4}(5)\to 0
\]
Of course $h^0(\mathbb E'_{\widehat{\pazocal V}/\hP})\simeq\mathbb T_{\widehat{\pazocal V}/\hP}$, but the previous lemma and the difference between $\mathcal H$ and $\mathcal K$ hint at the fact that the obstruction sheaf $h^1 (\mathbb E'_{\widehat{\pazocal V}/\hP})$ contains one factor $\R^1\hat\pi_{\widehat{\pazocal V},*}\OO_{\hC_{\widehat{\pazocal V}}}\simeq \OO_{\widehat{\pazocal V}}$ too many, so that restricting $\mathbb E'_{\widehat{\pazocal V}/\hP}$ to the fibers we would not find their standard obstruction theory. A confirmation of this fact is given by observing that $\mathbb E'_{\widehat{\pazocal V}/\hP}$ equips $\widehat{\pazocal V}$ with a $0$-dimensional cycle, while we are looking for a $1$-dimensional cycle such that restricting to any fiber, i.e. applying $\iota_t^!,$ we get $[\widehat{\pazocal V}_t]^{\rm{vir}}\in A_0(\widehat{\pazocal V}_t)$.

This issue is solved in \cite[\S\S4.5-6]{CLpfields} by lifting the standard obstruction theory $\phi'_{\widehat{\pazocal V}/\hP}$ to a different, specifically tailored one: 
\[\phi_{\widehat{\pazocal V}/\hP}\colon\mathbb T_{\widehat{\pazocal V}/\hP}\to\mathbb E_{\widehat{\pazocal V}/\hP},\]
 where the latter two-term complex fits into an exact triangle:
\[ \R^1\hat\pi_{\widehat{\pazocal V},*}\OO_{\hC_{\widehat{\pazocal V}}}[-2]\to\mathbb E_{\widehat{\pazocal V}/\hP}\xrightarrow{\nu}\mathbb E'_{\widehat{\pazocal V}/\hP}\xrightarrow{[1]}\]
Furthermore $\phi_{\widehat{\pazocal V}/\hP}$ is compatible with the standard obstruction theory for the fibers:
\begin{lem}
For every $t\in\mathbb A^1_\k$ we have a commutative diagram:
\bcd
\hat{\pi}_{\widehat{\pazocal V}_c *}\OO_{\hC_{\widehat{\pazocal V}_c}}[-1]\ar[r] & \mathbb E_{\widehat{\pazocal V}_c/\hP} \ar[r] &\mathbb E_{\widehat{\pazocal V}/\hP}\rvert_{\widehat{\pazocal V}_c}\\
\mathbb{T}^{\leq 1}_{\widehat{\pazocal V}_c/\widehat{\pazocal V}}\ar[u]\ar[r] & \mathbb{T}^{\leq 1}_{\widehat{\pazocal V}_c/\hP}\ar[u]\ar[r] & \mathbb{T}^{\leq 1}_{\widehat{\pazocal V}/\hP}\rvert_{\widehat{\pazocal V}_c} \ar[u]
\ecd
\end{lem}
For the proof see \cite[\S~4.6]{CL}.
Notice that the lower row is not a distinguished triangle, yet Chang and Li prove in \cite[A.4]{CLpfields} that this weaker (truncated) notion of compatibility is enough to prove functoriality.

Since we are working with cosection localised cycles, we need a family version of the cosection. This is induced by differentiating the following vector bundle morphism on $\hP\times\Aaff^1_t$:
\[
 \rm{Vb}(\hL_{\hP}^{\oplus 5}\oplus\hL_{\hP}^{\otimes 5}\oplus\widehat{\mathcal P}_{\hP})\xrightarrow{(\pr_2,\pr_3)}\rm{Vb}(\hL_{\hP}^{\otimes 5}\oplus\widehat{\mathcal P}_{\hP})\xrightarrow{\cdot}\rm{Vb}(\omega_{\hat\pi_{\hP}})
\]
The cosection takes then the following form:
\begin{align*}
\rm{Ob}_{\widehat{\pazocal V}/\hP\times\mathbb A^1}\subseteq \R^1\hat{\pi}_*(\hL_{\widehat{\pazocal V}}^{\oplus 5}\oplus\hL_{\widehat{\pazocal V}}^{\otimes}\oplus\widehat{\mathcal{P}}_{\widehat{\pazocal V}})\to \R^1\hat{\pi}_*\omega_{\hat{\pi}}\\
\bar\sigma_{1|(u,v,\psi)}(\mathring{x},\mathring{y},\mathring{p})=\psi\mathring{y}+v\mathring{p}
\end{align*}
It is showed in \cite[\S4.7]{CLpfields} that $\bar\sigma_1$ lifts to a cosection $\bar\sigma\colon \rm{Ob}_{\widehat{\pazocal V}}\to \OO_{\widehat{\pazocal V}}$ and that the degeneracy locus of $\bar\sigma$ is 
\[\Mone{1}{X}{d}\times\Aaff^1_t.\]

Recall that the sections $(u,v)$ are required to satisfy $\w(u)-tv=0$. So $\bar\sigma_1$ coincides up to a non-zero scalar with the above defined cosection $\sigma_1$ for $\widehat{\pazocal P}$ when $t\neq 0$. It is proved in \cite[Theorem 5.2]{KLcosection} that, given compatible perfect obstruction theories, the construction of a cosection localised virtual cycle is compatible with Gysin pullbacks, so that by \cite[Proposition 4.9]{CLpfields}:
\[
 \iota_{t\neq0}^![\widehat{\pazocal V}]^{\rm{vir}}_{\rm{\bar\sigma}}=[\widehat{\pazocal P}]^{\rm{vir}}_{\rm{\sigma}}\in A_0(\widehat{\pazocal Q}),\quad \iota_{0}^![\widehat{\pazocal V}]^{\rm{vir}}_{\rm{\bar\sigma}}=[\Mone{1}{N_{X/\PP^4}}{d}^p]^{\rm{vir}}_{\rm{\bar\sigma_0}}\in A_0(\widehat{\pazocal Q})
\]
where we have denoted by $\widehat{\pazocal Q}:=\Mone{1}{X}{d}$.
\subsubsection*{Functoriality of cosection-localised pullback}
We prove that $\deg[\Mone{1}{N_{X/\PP^4}}{d}^p]^{\rm{vir}}_{\rm{\bar\sigma_0}}$ coincides up to a sign with $\deg[\Mone{1}{X}{d}]^{\rm{vir}}$.
In the rest of the section we use the following notation: $\widehat{\pazocal N}:=\Mone{1}{N_{X/\PP^4}}{d}^p$, and $\hat v\colon\widehat{\pazocal N}\to\widehat{\pazocal Q}$. 

First, Chang and Li prove that there is a perfect obstruction theory 
\[\mathbb E_{\widehat{\pazocal N}/\widehat{\pazocal Q}}:=\R^{\bullet}\hat\pi_{\widehat{\pazocal N},*}(\hL^{\otimes 5}_{\widehat{\pazocal N}}\oplus \widehat{\mathcal P}_{\widehat{\pazocal N}})\]
 compatible with $\mathbb E_{\widehat{\pazocal N}/\hP}$ and $\hat v^*\mathbb E_{\widehat{\pazocal Q}/\hP}$, so that $\mathbb E_{\widehat{\pazocal N}/\widehat{\pazocal Q}}$ inherits a cosection $\sigma_0'$ with degeneracy locus $D(\sigma_0')=\widehat{\pazocal Q}$. 
So they have a localised virtual pullback:
\[
 \hat v^!_{\mathbb E_{\widehat{\pazocal N}/\widehat{\pazocal Q}},\rm{loc}}\colon A_*(\widehat{\pazocal Q})\to A_*(D(\sigma_0')=\widehat{\pazocal Q})
\]
which they then prove to satisfy functoriality using the techniques of \cite{KKP}, so:
\[\hat v^!_{\mathbb E_{\widehat{\pazocal N}/\widehat{\pazocal Q}},\rm{loc}}[\widehat{\pazocal Q}]^{\rm{vir}}=[\Mone{1}{N_{X/\PP^4}}{d}^p]^{\rm{vir}}_{\rm{\sigma'_0}}.\]
To conclude it is  enough to compute the degree of $\hat v^!_{\mathbb E_{\widehat{\pazocal N}/\widehat{\pazocal Q}},\rm{loc}}$
on $A_0(\widehat{\pazocal Q})$; we just need to compute $\deg(\hat v^!_{\mathbb E_{\widehat{\pazocal N}/\widehat{\pazocal Q}},\rm{loc}}[\zeta])$ for a closed point $\zeta$. This is done in \cite[Theorem 5.7]{CLpfields} and the same considerations work in our case.

Hopefully we have managed to convince the reader that the subtle intersection theory perpetuated in \cite{CLpfields} does not rely at all on the hypothesis that the families of curves we are working with are \emph{nodal}, but \emph{l.c.i.} curves are well-behaved enough so that all the proofs carry over to the situation of our interest.

\section{The weighted $1$-stabilisation morphism}\label{sec:comparison}

Before stating the main result of this section we recall the following:
\begin{dfn}
$\mathfrak M_{1,n}^{\rm{wt}=d,\rm{st}}$ denotes the stack of \emph{weighted-stable} curves of genus $1$ with $n$ markings and total weight $d$: geometric points of $\mathfrak M_{1,n}^{\rm{wt}=d,\rm{st}}$ represent connected, reduced, nodal, projective curves of arithmetic genus $1$ with $n$ distinct smooth markings and an integer-valued function $\rm{wt}$ on the dual graph. 

Such a function assumes only nonnegative values, the sum of which on all the vertices of the dual graph is $d$; $\rm{wt}$ is compatible with the specialisation maps, and we further impose the following stability condition: 
\begin{itemize}
\item every $p_a=0$ component of weight $0$ has at least three special points;
\item every $p_a=1$ component of weight $0$ has at least one special point.
\end{itemize}
\end{dfn}
There is an \'{e}tale, non-separated morphism $\mathfrak M_{1,n}^{\rm{wt}=d,\rm{st}}\to\mathfrak M_{1,n}$.
The stability condition is such that the forgetful map $\M{1}{n}{\PP^r}{d}\to\mathfrak M_{1,n}$ factors through $\mathfrak M_{1,n}^{\rm{wt}=d,\rm{st}}$, the weight coming from the degree of the map to $\PP^r$.

\begin{dfn}
Let $\mathfrak M_{1,n}^{\operatorname{wt}=d,\text{st}}(1)$ be the stack of \emph{at worst cuspidal} connected, reduced, $n$-marked, projective curves of arithmetic genus $1$ that are weighted-stable with total weight $d$, i.e. the weight is nonnegative and such that:
\begin{itemize}
\item every $p_a=0$ component of weight $0$ has at least three special points; 
\item every $p_a=1$ component of weight $0$ has at least \emph{two} special points.
\end{itemize}
In this definition, by \emph{special point} we mean the preimage of a node or a marking in the normalisation. 

\end{dfn}
\begin{remark}
 It follows from the miniversal deformation of the cusp that being at worst cuspidal (i.e. having only nodes and cusps as singolarity) is an open condition on the base of any family of curves.
\end{remark}
\begin{thm}
There exists a morphism $\mathfrak M_{1,n}^{\operatorname{wt}=d,\text{st}}\to\mathfrak M_{1,n}^{\operatorname{wt}=d,\text{st}}(1)$ which extends the identity on the smooth locus.
\end{thm}
As anticipated we explain two different approaches to the proof:
\begin{enumerate}
 \item we first construct a substack of the product $\mathfrak M_{1,n}^{\operatorname{wt}=d,\text{st}}\times\mathfrak M_{1,n}^{\operatorname{wt}=d,\text{st}}(1)$ which plays the role of the graph of such a morphism, and then check that the projection onto the first factor is an isomorphism;
\item we prove that the $1$-stabilisation exists at the level of curves with a divisor constructing the contraction directly, then argue that it descends to a morphism between moduli spaces of weighted curves.
\end{enumerate}

\subsection{First approach: the graph}

Recall the shorthand notation $\MM=\mathfrak M_{1,n}^{\operatorname{wt}=d,\text{st}}$ and $\hM=\mathfrak M_{1,n}^{\operatorname{wt}=d,\text{st}}(1)$, with universal curves $\pi\colon \cC\to \MM$ and $\hat\pi\colon\hC\to\hM$ respectively. Abusing notation we will still write $\cC$ and $\hC$ for their pullbacks to the product $\MM\times\hM$ along the two projections.

\begin{lem}\label{lemma:def_X}
 There is a locally closed substack $\X\subseteq\operatorname{Mor}_{\MM\times\hM}(\cC,\hC)$ representing morphisms that contract genus $1$ tails of weight $0$ to cusps, and are weight-preserving isomorphisms everywhere else. 
\end{lem}

\begin{proof}
Recall that $\operatorname{Mor}_{\MM\times\hM}(\cC,\hC)$ is an algebraic stack; in fact the map to $\MM\times\hM$ is representable by algebraic spaces \cite{OlssonHOM}. Let $\phi\colon\cC\to\hC$ denote the universal morphism. We now proceed to construct $\X$ as a locally closed substack in the space of morphisms.

\begin{description}[labelindent=0cm,leftmargin=\parindent]

\item[Step 1] We first lift the problem to the level of Picard stacks in order to transmute the property of being weight-preserving into the more manageable one of preserving the line bundles. Recall the notation:
 \[\PtoM\colon\mathfrak P=\mathfrak{Pic}^{tot deg=d,st}_{1,n}\to\MM,\quad \hat\PtoM\colon\hP=\mathfrak{Pic}^{tot deg=d,st}_{1,n}(1)\to\hM\]
  for the Picard stacks of $\pi$ and $\hat \pi$ respectively, with universal line bundles $\L$ and $\hL$. We can now look at the algebraic stack $\operatorname{Mor}_{\mathfrak{P}\times\hP}(\cC,\hC)$ with universal morphism
 $\Phi$ and natural projection $\Pi$ to $\operatorname{Mor}_{\MM\times\hM}(\cC,\hC)$. There exists a locally closed substack $\pazocal{Y}'\subseteq \operatorname{Mor}_{\mathfrak{P}\times\hP}(\cC,\hC)$ representing those morphisms that \emph{preserve the line bundles}. Indeed, given a smooth chart $S\to\operatorname{Mor}_{\mathfrak{P}\times\hP}(\cC,\hC)$, the locus of $s\in S$ where $\Phi_s^*\hL_s\cong\mathcal{L}_s$ is nothing but the locus $T$ where the two sections $\mathcal{L}_S$ and $\Phi_S^*\hL_S$ of $\mathfrak{P}(S)\to\MM(S)$ are isomorphic. In other words we are looking at the fiber product 
  \[\xymatrix{T\ar[r]\ar[d]&\mathfrak{P}\ar[d]^{\Delta}\\
  S\ar[r]&\mathfrak{P}\times_{\MM}\mathfrak{P}}\]
Being $\mathfrak{P}\to\MM$ representable by locally separated algebraic spaces \cite[Theorem~8.3.1]{neron}, $\Delta$ is a quasi-compact locally closed immersion \cite[\href{http://stacks.math.columbia.edu/tag/04YU}{Tag 04YU}]{stacks-project}, so in particular $T\subseteq S$ is locally closed.

\item[Step 2] 
Furthermore there is a closed substack $\pazocal{Y}\subseteq \pazocal{Y}'$ representing \emph{surjective morphisms that preserve the markings}.

Given a chart $S\to\operatorname{Mor}_{\mathfrak{P}\times\hP}(\cC,\hC)$, the locus of $s\in S$ where $\Phi_s$ is marking-preserving is the equaliser of the two sections:
\bcd
S \ar[r,shift left, "\times\hat\sigma_i" above]\ar[r,shift right,"\times\Phi\circ\sigma_i" below] & \hC_S\times_S\ldots\times_S \hC_S
\ecd
This defines a closed subscheme of $S$, since $\hC_S\to S$ is separated.

As regards surjectivity, since $\Phi$ is proper and the dimension of the fiber is upper semicontinuous \cite[\href{http://stacks.math.columbia.edu/tag/0D4I}{Tag 0D4I}]{stacks-project}, the locus in $\hC_S$ where the fiber of $\Phi$ is empty is open. Its image in $S$ is open by flatness of $\hC_S\to S$ \cite[\href{http://stacks.math.columbia.edu/tag/01UA}{Tag 01UA}]{stacks-project}, and the complement of the latter is the locus we need.
 
\item[Step 3] We may now get back to $\operatorname{Mor}_{\MM\times\hM}(\cC,\hC)$. Let $\X'$ be the image of $\pazocal{Y}$ under $\Pi$. This is a constructible substack of $\operatorname{Mor}_{\MM\times\hM}(\cC,\hC)$ by Chevalley's theorem \cite[Theorem 5.9.4]{LMB}. Recall that to show that a constructible set is open (respectively closed) it is enough to check that it contains all the generisations of its points (respectively all the specialisations) \cite[\href{http://stacks.math.columbia.edu/tag/0DQN}{Tag 0DQN}\href{http://stacks.math.columbia.edu/tag/0903}{Tag 0903}]{stacks-project}. Finally, under Noetherian assumptions, two points related by specialisation/ generisation are contained in the image of a DVR scheme \cite[\href{http://stacks.math.columbia.edu/tag/054F}{Tag 054F}]{stacks-project}.

It can be shown as above that being surjective and marking-preserving are closed conditions. The requirement that $\phi$ can be covered by a line bundle-preserving morphism can be translated into the following combinatorial conditions:

\begin{enumerate}
\item\emph{$\phi$ contracts only components of weight $0$.} We show that this is open. Assume that $\dvr$ is a DVR scheme with closed point $0$ and generic point $\eta$, and we are given $\dvr\to\X'$ such that $\phi_0\colon\cC_0\to \hC_0$ does not contract any component of positive weight. Suppose there exists an irreducible component $D_\eta\subseteq \cC_{\eta}$ of positive weight $d_D$ which is contracted by $\phi_{\eta}.$ The contracted locus, i.e. $\left\{c\in \cC_{\dvr} | \dim_c\phi^{-1}(\phi(c))\geq 1\right\}$, is closed by semicontinuity of fiber dimension, hence it contains all the components $D_i\subseteq \cC_0$ to which $D_\eta$ specialises. At least one of them has positive weight, since the sum of their weights is $d_D$, which is a contradiction.

\item\emph{$\phi$ has degree 1 on every non contracted component} or, equivalently, for every $S\to\X'$ there is an $S$-dense open in $\hC_S$ such that the restriction of $\phi_S$ to its preimage is an isomorphism.
This is an open and closed condition; we show it is open. Let $\dvr$ be a DVR scheme as above and assume that $\phi_0$ satisfies the property. Since $\phi_\dvr$ is proper, we may consider 
\[\phi_{\dvr,*}[\cC_{\dvr}]=\sum n_i[\hC_{\dvr,i}]\in A_2(\hC_\dvr).\] Applying Gysin pull-back to $0$ (which is a regular closed point of the base) \cite[Prop. 10.1(a)]{FUL}, we see that all the $n_i$'s are 1 for those $\hC_i$'s such that $0^![\hC_i]\neq 0$. On the other hand there is no irreducible component of $\hC$ supported on $\hC_\eta$.

\item\emph{$\phi$ is weight-preserving.} This is again an open condition, as we can see from the weighted dual graphs. Let $\Gamma(\phi)$ be the map induced at the level of weighted dual graphs $\Gamma(\cC_\dvr)\to\Gamma(\hC_\dvr)$. It is compatible with the specialisation maps:
\bcd
\Gamma(\cC_0)\ar[r,"\Gamma(\phi_0)"]\ar[d,"\rm{sp}"] & \Gamma(\hC_0)\ar[d,"\rm{sp}"] \\
\Gamma(\cC_{\eta})\ar[r,"\Gamma(\phi_\eta)"]& \Gamma(\hC_\eta)
\ecd
Since the weight of a component of the generic fiber is determined by those of the components to which it specialises
\[ \deg(v)=\sum_{w\in\rm{sp}^{-1}(v)}\deg(w)
\]
$\Gamma(\phi_\eta)$ has to be weight-preserving as well.
\end{enumerate}

\item[Step 4] We have imposed that, if $\phi$ contracts a subcurve $E$ of the fiber, it must have weight $0$. Since the fibers of $\hC$ only have nodes and cusps as singularities, and the markings are required to be smooth points, we observe that $E$ must be unmarked and have arithmetic genus $1$ by weighted stability; furthermore $\left|E\cap\overline{C\setminus E}\right|\leq 2$, i.e. $E$ is either an elliptic tail or an elliptic bridge. There are two possibilities left:
\begin{enumerate}
\item $\phi$ contracts an elliptic tail to a cusp and is an isomorphism everywhere else, or there is no elliptic tail to start with and $\phi$ is an isomorphism;
\item the elliptic tail/bridge is contracted to a smooth point/node, then a non-separating node or a cusp must appear somewhere else in the fiber of $\hC$ in order to preserve the arithmetic genus.
\end{enumerate} 
We want to avoid the second scenario, so we define the open substack $\X\subseteq \X'$ as follows.
Given $\mathcal{C}_S\to\hC_S\in\X'(S)$, let $U\subseteq \hC$ be the \emph{maximal} $S$-dense open subset such that $\phi_S|_{\phi_S^{-1}(U)}\colon\phi_S^{-1}(U)\to U$ is an isomorphism, and $V$ its closed complement in $\hC_S$. 
Then $\X$ is the open locus \cite[\href{http://stacks.math.columbia.edu/tag/055G}{Tag 055G}]{stacks-project} where the fibers of $\pi|_{\phi^{-1}(V)}\colon \phi^{-1}(V)\to S$ are geometrically connected.
\end{description} 
This concludes the construction of the locally closed substack $\X\subseteq \operatorname{Mor}_{\MM\times\hM}(\cC,\hC)$.
 \end{proof}

\begin{lem}\label{lamma:projection_iso}
 The first projection $\pr_1\colon \operatorname{Mor}_{\MM\times\hM}(\cC,\hC)\to\MM$ restricted to $\X$ is an isomorphism with $\MM$.
\end{lem}

 \begin{proof}
 This result will follow from an application of Zariski's Main Theorem for algebraic spaces.
First we claim that the projection $\pr_{1|\X}\colon\X\to \MM$ is \emph{representable by algebraic spaces}: by \cite[\href{http://stacks.math.columbia.edu/tag/04Y5}{Tag 04Y5}]{stacks-project} we only need to check that it is faithful, and by \cite[Theorem 2.2.5]{CONR} it is enough to look at geometric points. Hence we need to say that, given $\phi\colon C\to \widehat{C}$ a $K$-point of $\X$, we have
$\operatorname{Aut}(\phi)\subseteq\operatorname{Aut}(C)$. Recall that automorphisms of $\phi$ are commutative diagrams:
\bcd C\ar[r,"\phi"]\ar[d,"\psi"] & \widehat{C}\ar[d,"\hat\psi"] \\ C\ar[r,"\phi"] & \widehat{C} \ecd
 Now $\hat\psi$ is determined by $\psi:$ for any $p\in\widehat{C}$, either $\phi$ is an isomorphism in a neighbourhood of its preimage and thus around that point $\hat{\psi}=\phi\circ\psi\circ\phi^{-1},$ or $\phi^{-1}(p)$ is a genus 1 tail, so $\psi$ must preserve it.
 
Secondly $\pr_{1|\X}$ is \emph{proper}: this can be seen using the valuative criterion:
\bcd
\eta'=\operatorname{Spec}(K')\ar[r]\ar[d] & \eta=\operatorname{Spec}(K)\ar[r]\ar[d] & \X \ar[d] \\
\dvr'=\operatorname{Spec}(R')\ar[r]\ar[urr,dashed,"\exists?", left] & \dvr=\operatorname{Spec}(R)\ar[r] & \MM
\ecd
Let $\pi \colon \cC_\dvr\to \dvr$ be the family of nodal curves on $\dvr$; there are three cases to consider:
\begin{enumerate}[label=(\alph*)]
\item the central fiber contains no elliptic tail, then the same is true for $\cC_\eta$, hence $\phi_\eta$ is an isomorphism. We can extend $\phi_\eta$ as follows:
\bcd
\cC_\eta\arrow{r}{\sim}[swap]{\phi_\eta}\ar[d,hook,"\iota"] & \hC_\eta\ar[d,hook,"\iota\circ\phi_\eta^{-1}"] \\
\cC_\dvr\ar[r,"\id_\cC"] & \cC_\dvr=:\hC_\dvr
\ecd
Another extension $\phi'\colon \cC_\dvr\cong \hC_\dvr$ would be isomorphic to the previous one via:
\bcd
\cC_\dvr\ar[r,"\phi'"]\ar[d,"\id"] &\hC_\dvr\ar[d,"(\phi')^{-1}"] \\
\cC_\dvr\ar[r,"\id"] & \cC_\dvr
\ecd

\end{enumerate}
If instead $\cC_0$ has got an elliptic tail, then we have two possibilities:
\begin{enumerate}[label=(\alph*)]
\setcounter{enumi}{1}
\item $\cC_{\eta}$ has got an elliptic tail as well; that is the image of $\dvr\to\MM$ is contained in the boundary, so we can find a lift
\bcd
 & \MM_{1,1}\times\MM_{0,1+n}^{\rm{wt}=d,\rm{st}}\ar[d] & \\
 \dvr\ar[r]\ar[ur,dashed] & \mathfrak D_{\{1,\emptyset\},\{0,n\}}\ar[r,hook] &\MM.
\ecd
 Then $\cC_\dvr$ is the pushout of a family of rational curves $\mathcal R_\dvr$ and a family of genus 1 curves $\mathcal E_\dvr$:
\begin{equation}\label{eq:pushout}
\begin{tikzcd}
\dvr\ar[d]\ar[r] & \mathcal R_\dvr \ar[d] \\
\mathcal E_\dvr\ar[r] & \cC_\dvr
\end{tikzcd}
\end{equation}
Recall that the cuspidal curve $\hC_K$ can be described as the pushout of the following diagram:
\bcd
2K\ar[r]\ar[d] & \mathcal R_K\ar[d] \\
K\ar[r] & \hC_K.
\ecd
Since the smooth section $\dvr\to\mathcal R_\dvr$ defines a Cartier divisor, it makes sense to take its double and we can thus define $\hC_\dvr$ by means of the similar diagram:
\bcd
2\dvr\ar[r]\ar[d] & \mathcal R_\dvr\ar[d] \\
\dvr\ar[r] & \hC_\dvr
\ecd
The morphism $\phi_\dvr\colon\cC_\dvr\to\hC_\dvr$ extending $\phi_\eta$ is then defined by exploiting the description of $\cC_\dvr$ as a pushout \eqref{eq:pushout}, and the morphisms $\id\colon\mathcal R_\dvr\to \mathcal R_\dvr$ and $\pr_{\mathcal E_\dvr}\colon\mathcal E_\dvr\to \dvr$ to the upper right and bottom left corners of the last diagram respectively.
\item If $\cC_\dvr$ does smooth the elliptic tail, then $\phi_\eta$ is an isomorphism. We may assume that $\dvr$ is the spectrum of a complete DVR with algebraically closed residue field \cite[Theorem 7.10]{LMB}. Then we may pick a smooth section for each rational component of $\cC_0$ and extend them to sections of $\cC_\dvr\to \dvr$ by Grothendieck's existence theorem; let us denote by $\Sigma$ the Cartier divisor that is the sum of all such sections. Let $Z$ be the elliptic tail in the central fiber; then we claim that $\omega_{\cC_\dvr/\dvr}(Z)\otimes\mathcal O_{\cC_\dvr}(2\Sigma)$ is $\pi_\dvr$ semi-ample, ample on the generic fiber, and gives the contraction of the elliptic tail to the cusp in the central fiber. We shall not prove the claim here, since this is the core of the second approach.
\end{enumerate}
Finally observe that the map is bijective by construction and $\MM$ is normal, hence $\pi_{|\X}\colon\X\to \MM$ is an isomorphism by Zariski's main theorem (as in \cite[\href{http://stacks.math.columbia.edu/tag/082I}{Tag 082I}]{stacks-project}).
 \end{proof}
 
\subsection{Second approach: constructing the contraction}

The idea behind this construction is essentially due to Hassett \cite[\S2]{HassettHyeon} and it has recently been reviewed and simplified in \cite[\S3.7]{RSPW}. For simplicity we will work with unmarked curves, which is indeed the case of interest when looking at the Gromov-Witten theory of a Calabi-Yau threefold.

We shall construct the contraction over $\MM_1^{\rm{div}}$ first, and then show that it descends to $\MM_1^{\rm{wt,st}}$. The weighted stability condition is implicit in the notation.

Let $\pazocal E$ be the locus inside the universal curve spanned by elliptic tails of weight $0$; this is a Cartier divisor in the universal curve over $\MM^{\rm{wt,st}}_1$; we will abuse notation and denote by $\pazocal E$ all its pullbacks. Moreover we denote by $\mathfrak D^1$ its image in $\MM^{\rm{wt,st}}_1$, which is a Cartier divisor as well.

Consider the following line bundle on the universal curve over $\MM_1^{\rm{div}}$: 
\begin{equation}\label{eq:linebundlecontraction}
\mathcal N:=\omega_{\pi}(\pazocal E)\otimes\mathcal O_{\cC}(2\pazocal D),
\end{equation} where $\pazocal D$ is the universal Cartier divisor over $\MM_1^{\rm{div}}$.
Notice that $\mathcal N$ is trivial on the locus of elliptic tails, so the Proj construction applied to $\mathcal N$ will contract this locus.

\begin{prop}\label{1-stabilization-div}
Let $\hC=\underline{\operatorname{Proj}}_{\MM_1^{\rm{div}}}(\bigoplus_{n\geq 0}\pi_*\mathcal N^{\otimes n})$. Then $\hC$ is a family of weighted $1$-stable curves and $\phi$ is a regular morphism:
 
 \bcd
 (\cC,\pazocal D)\ar[rr,"\phi"]\ar[dr,"\pi"] & & (\hC,\phi(\pazocal D))\ar[ld,"\hat\pi" above left] \\
 & \MM_1^{\rm{div}} &
 \ecd
This defines the $1$-stabilisation morphism $\MM_1^{\rm{div}}\to\MM_1^{\rm{div}}(1)$.
\end{prop}

We need to prove that $\mathcal N$ is $\pi$-semi-ample (regularity of $\phi$) and that $\pi_*\mathcal N$ is locally free (flatness of $\hat\pi$). Both these facts are clearly true generically, but less so on points of $\pazocal E$ and $\mathfrak D^1$. We shall check this by exploiting the next lemma, which is a nice technical gadget taken from \cite{RSPW}.
\begin{lem}[pullback with a boundary]\label{DVR}
Let $\pi\colon\cC\to S$ be a proper family of curves over a smooth basis, and let $\mathcal N$ be a line bundle on $\cC$ such that $\R^1\pi_*\mathcal N$ is a line bundle supported on a Cartier divisor $\mathfrak D\subseteq S$. Then for every DVR scheme $\dvr$ with closed point $0$ and generic point $\eta$, and for every morphism $f\colon \dvr\to S$ such that $f(0)\in\mathfrak D$ and $f(\eta)\in S\setminus\mathfrak D$ we have
\[f^*\pi_*\mathcal N\cong \pi_{\dvr,*}f_\cC^*\mathcal N.\]
\end{lem}
\begin{proof}
The argument can be found in \cite[Lemmma~3.7.2.2]{RSPW}.
\end{proof}
Now recall that in our case $\mathcal N$ is trivial on elliptic tails and of positive degree everywhere else. The rank of $\R^1\pi_*\mathcal N$ can be checked on the fibers \cite[Theorem III.12.11]{HAR}, so we see that it is $0$ outside $\mathfrak D^1$ and $1$ on it.

\begin{lem}\label{lemma:semiample}
The line bundle $\mathcal N$ is $\pi$-semi-ample, i.e. the natural map
\[\pi^*\pi_*\mathcal N^{\otimes n}\to \mathcal N^{\otimes n}\]
is surjective for $n\gg 0$.
\end{lem}
\begin{proof}
Outside the locus of elliptic tails $\mathcal N$ is $\pi$-ample. We are left to check on points belonging to an elliptic tail; thanks to the above Lemma we can reduce to the case that $C$ is the central fiber of a one-parameter smoothing of the elliptic tail. This has been dealt with by Smyth~\cite[Lemma~2.12]{SMY1}.
\end{proof}

\begin{lem}
$\pi_*\mathcal N$ is locally free on $\MM_1^{\rm{div}}.$
\end{lem}
\begin{proof}\cite[Proposition~3.7.2.1]{RSPW}
We have to check that $\pi_*\mathcal N$ has constant rank.
On $\MM_1^{\rm{div}}\setminus \mathfrak{D}^1$ we see that $\R^1\pi_*\mathcal N=0$, so $\pi_*\mathcal N$ satisfies Cohomology and Base Change and its rank is determined by Riemann-Roch.
Given a point $x$ on the boundary $\mathfrak{D}^1$, we can pick a DVR scheme $\dvr$ whose closed point maps to $x$ and whose generic point maps to $\MM_1^{\rm{div}}\setminus \mathfrak{D}^1.$ Then we are in the hypotheses of Lemma~\ref{DVR} and we can check the rank at $x$ by looking at $\pi_*f^*\mathcal N$ over $\dvr.$ Now $f^*\mathcal N$ is flat over $\dvr$, so $\pi_*f^*\mathcal N$ is as well, which implies torsion-free and thus constant rank.
\end{proof}

\begin{proof}\ref{1-stabilization-div}
Let $S\to \MM_1^{\rm{div}}$ be a smooth atlas, then we have:
 \bcd
 (\cC_S,\pazocal D_S)\ar[rr,"\phi_S"]\ar[dr,"\pi_S"] & & (\hC_S,\phi(\pazocal D_S))\ar[ld,"\hat\pi_S" above left] \\
 & S &
 \ecd
where $\hC_S=\underline{\operatorname{Proj}}_{S}(\bigoplus_{n\geq 0}\pi_{S,*}\mathcal N^{\otimes n})$, $\phi_S$ is a proper and birational morphism since $\mathcal N$ is $\pi_S$-semi-ample and $\hat{\pi}_S$ is flat since $\pi_{S,*}\mathcal N$ is locally free. 
To verify that this defines a morphism $S\to \MM_1^{\rm{div}}(1)$ we have to argue that $\hC_S$ has reduced fibers and only nodes and cusps as singularities.

Since these properties only concern the fibers of $\hat{\pi}$ we can verify them after base change to a DVR scheme $\dvr$ chosen as in Lemma~\ref{DVR}, so that the construction commutes with base-change. Furthermore we can pick $f\colon \dvr\to S$ such that the total space $\cC_\dvr$ is regular, so we may just apply Smyth's Contraction Lemma~\cite[Lemma~2.13]{SMY1}.

Finally to conclude that this defines a morphism $\MM_1^{\rm{div}}\to \MM_1^{\rm{div}}(1)$ it is enough to verify that there is an isomorphism
 $\rm{pr}_1^*\hC_S\cong\rm{pr}_2^*\hC_S$ satisfying the cocycle condition, where $\rm{pr}_i\colon S'= S\times_{\MM_1^{\rm{div}}} S\rightrightarrows S$.

 This follows from the fact that $\pr_i^*\hC_S$ are obtained from applying the Proj construction to 
 $\rm{pr}_i^*\pi_{S,*}\mathcal{N}\cong\pi_ {S',*} \rm{pr}_i^*\mathcal N$, by flatness of $S'\to S$. Thus it is enough to show that $\rm{pr}_1^*\mathcal N\cong \rm{pr}_2^*\mathcal N$.
  But $\mathcal N$ is the pullback of a line bundle on $\MM_1^{\rm{div}}$, thus the desired isomorphism follows from the commutativity of the diagram
 
 \bcd
S\times_{\MM_1^{\rm{div}}} S\ar[r,"\rm{pr_1}"]\ar[d, "\rm{pr_2}"] & S\ar[d] \\
S\ar[r] & \MM_1^{\rm{div}}
\ecd
 
The cocycle condition is derived similarly. 
\end{proof}

\begin{prop}
The $1$-stabilisation for curves with a divisor induces an analogue morphism on weighted curves:
\bcd
\MM_1^{\rm{div}}\ar[r]\ar[d]\ar[dr,phantom,"\circlearrowright"] & \MM_1^{\rm{div}}(1)\ar[d] \\
\MM_1^{\rm{wt,st}}\ar[r,"\exists"] & \MM_1^{\rm{wt,st}}(1)
\ecd
\end{prop}
\begin{proof}
{\'E}tale locally on $\MM_1^{\rm{wt,st}}$ we can choose smooth sections $s_i$ of the universal curve so that the Cartier divisor $\pazocal D=\sum s_i$ has degree compatible with the weight function, so in particular it makes $\mathcal N=\omega_{\pi}(\pazocal E)\otimes\mathcal O_{\cC}(2\pazocal D)$ trivial on the elliptic tails and $\pi$-ample elsewhere.
For a smooth atlas $S\to \MM_1^{\rm{wt,st}}$, this observation allows us to define a lifting $S\to \MM_1^{\rm{div}}$, and thus a morphism $\xi\colon S\to  \MM_1^{\rm{wt,st}}(1)$ through the construction of Proposition \ref{1-stabilization-div}.

In order to show that this descends to a morphism $\MM_1^{\rm{wt,st}}\rightarrow \MM_1^{\rm{wt,st}}(1)$ we need to verify that there exists $\rm{pr}_1^*(\xi)\cong\rm{pr}_2^*(\xi)$ satisfying the cocycle condition, where $\rm{pr}_i\colon S'= S\times_{\MM_1^{\rm{wt,st}}} S\rightrightarrows S$.

This boils down to checking that for two different choices of a lifting $\pazocal D_1,\pazocal D_2\colon S\to\MM_1^{\rm{div}}$ there exists a unique isomorphism 
\[\hC_1=\underline{\operatorname{Proj}}_S\left(\bigoplus_{n\geq 0}\pi_*(\mathcal N_1))^{\otimes n}\right)\cong \underline{\operatorname{Proj}}_S\left(\bigoplus_{n\geq 0}\pi_*(\mathcal N_2))^{\otimes n}\right)=\hC_2. \]
By construction there is a birational map $\psi$:
 \bcd
&\cC_S \ar[dl,"\phi_1" above left]\ar[dr,"\phi_2"]  \\
 \hC_1\ar[rr,dashed, "\psi"] & &\hC_2.
 \ecd

We want to show that $\psi$ extends to a regular morphism. 
Notice that $ \hC_i$ is normal, $i=1,2$. Indeed since $S$ is smooth and the singularities of the fibers are in codimension $1$, $ \hC_i$ is regular in codimension $1$. Moreover since both $S$ (smooth) and the fibers (Cohen-Macaulay) satisfy Serre's condition $S_2$, so does the total space of $\hC_i$ by \cite[Thorem~23.9]{MAT}. Then, since the geometric fibers of $\phi_i$ are either points or $p_a=1$ subcurves, and in either case connected, Zariski's connectedness theorem implies that
\[\phi_{i,*}\OO_{\cC_S}\cong \OO_{\hC_i}. \]
 Moreover by construction $\operatorname{Exc}(\phi_1)=\operatorname{Exc}(\phi_2)$ is the locus of elliptic tails of weight $0$, so in particular $\phi_2$ contracts all fibers of $\phi_1.$ Then \cite[Lemma 1.15]{debarre}
implies that $\phi_2$ factors through $\phi_1$, and viceversa. This proves the regularity of $\psi$ and its inverse. Notice that $\psi$ is unique as it is the only extension of $\phi_2\circ\phi_1^{-1}$.
\end{proof} 
\begin{remark}
The $m$-stabilisation for $m\ge 2$ does not extend to a regular morphism; the following is an example with stable curves:

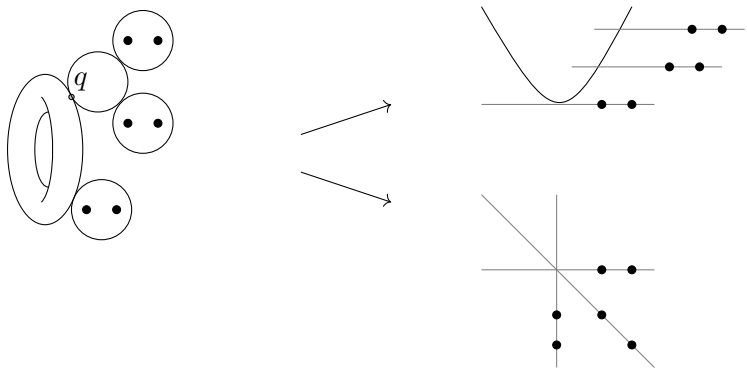
\begin{figure}[h]\label{fig:3-stab}
\centering
\begin{tikzpicture}
\draw (-.1,0) circle (.4cm);
\draw (.5,.55) circle (.4cm);
\node[left] at (-.1,0){$q$};
\draw[fill=gray!50] (-.45,-.2) circle[radius=1pt];
\draw[fill=black] (.3,.55) circle[radius=1.5pt];
\draw[fill=black] (.7,.55) circle[radius=1.5pt];
\draw (.5,-.55) circle (.4cm);
\draw[fill=black] (.3,-.55) circle[radius=1.5pt];
\draw[fill=black] (.7,-.55) circle[radius=1.5pt];
\draw (-.8,-.9) ellipse (.5cm and 1cm);

\draw (-.85,-.2) .. controls (-.65,-.4) and (-.65,-1.4) .. (-.85,-1.6);
\draw (-.75,-.4) .. controls (-1,-.4) and (-1,-1.4) .. (-.75,-1.4);
\draw (-.05,-1.7) circle (.4cm);

\draw[fill=black] (.15,-1.7) circle[radius=1.5pt];
\draw[fill=black] (-.25,-1.7) circle[radius=1.5pt];

\draw[->] (2.6,-.7) -- (3.8,-.3);
\draw[->] (2.6,-1.2) -- (3.8,-1.6);

\draw[gray] (5,-.3)-- (7.3,-.3);
\draw[fill=black] (7,-.3) circle[radius=1.5pt];
\draw[fill=black] (6.6,-.3) circle[radius=1.5pt];
\draw (5,1) .. controls (6,-.7) and (6.1,-.7) .. (7,1);
\draw[gray] (6.5,.7)-- (8.5,.7);

\draw[fill=black] (7.8,.7) circle[radius=1.5pt];
\draw[fill=black] (8.2,.7) circle[radius=1.5pt];
\draw[gray] (6.2,.2)-- (8.2,.2);
\draw[fill=black] (7.5,.2) circle[radius=1.5pt];
\draw[fill=black] (7.9,.2) circle[radius=1.5pt];

\draw[gray] (5,-2.5)-- (7.3,-2.5);
\draw[fill=black] (7,-2.5) circle[radius=1.5pt];
\draw[fill=black] (6.6,-2.5) circle[radius=1.5pt];
\draw[gray] (6,-1.5)-- (6,-3.8);
\draw[fill=black] (6,-3.5) circle[radius=1.5pt];
\draw[fill=black] (6,-3.1) circle[radius=1.5pt];
\draw[gray] (5,-1.5)-- (7.3,-3.8);
\draw[fill=black] (7,-3.5) circle[radius=1.5pt];
\draw[fill=black] (6.6,-3.1) circle[radius=1.5pt];
\end{tikzpicture}
\caption{An example of different plausible $3$-stabilisations.}
\end{figure}

Both curves on the right are $3$-stable, so it is unclear how to define the $3$-stabilisation already on points.

 Here is some heuristics about why the natural choices in families do not work. Indeed suppose we try contracting the minimal genus $1$ unmarked subcurve. Consider a $1$-parameter smoothing of the node $q$: the generic fiber is a smooth elliptic curve with three marked $\PP^1$ attached. If a contraction of the core existed, we would get a family with a Smyth's $3$-fold elliptic point degenerating to the tacnode, but such a family cannot exist; the miniversal deformation of the tacnode contains only cusps and nodes as singularities.

On the other hand, try to define the $3$-stabilisation by contracting the maximal unmarked genus $1$ subcurve. 
Choose a generic smoothing of our stable curve in such a way that the total space $\cC_{\dvr}$ is smooth. Notice that the maximal unmarked subcurve of the central fiber is not balanced~\cite[Definition 2.11]{SMY1}, so contracting it we would get a non-Goreinstein singularity.

The indeterminacy has been resolved by \cite{RSPW}.
\end{remark}
\subsection{Auxiliary spaces and induced obstruction theories}
In this last part of the section we exploit the $1$-stabilisation morphism to introduce some auxiliary spaces with a virtual class that are going to be useful in comparing cuspidal with reduced invariants of the quintic threefold. Let $\Z$ be defined by the pullback diagram:
\bcd
\Z \ar[r]\ar[d]\ar[rd,phantom,"\Box"] & \overline{\mathcal M}_1^{(1)}(\PP^4,d) \ar[d] \\
\MM_1^{\rm{wt,st}}\ar[r] & \MM_1^{\rm{wt,st}}(1)
\ecd

Objects of $\Z$ over a scheme $S$ consist of diagrams:
\bcd
C \ar[rr,"\phi"]\ar[dr,"\pi" below left] & & \widehat C\ar[rr,"f"]\ar[dl,"\hat\pi"] & & \PP^4 \\
& S & & &
\ecd
where $f$ is a $1$-stable map and $\phi$ is the weighted $1$-stabilisation; arrows over $\id_S$ are commutative diagrams:
\bcd
C\ar[r,"\phi"]\ar[d,"\psi"] &  \widehat C\ar[r,"f"]\ar[d,"\hat\psi"] & \PP^4 \ar[d,"\id_\PP"]\\
C'\ar[r,"\phi'"] &  \widehat C'\ar[r,"f'"] & \PP^4

\ecd
where $\psi$ and $\hat\psi$ are isomorphisms.  Recall that $\hat \psi$ is determined by $\psi$. 

Forgetting $\widehat C$ and keeping $f\circ\phi\colon C\to\PP^4$, we obtain a morphism \[i\colon\Z\to\oM_1(\PP^4,d).\]

\begin{lem}\label{lem:inclusion}
The morphism $i\colon\Z\hookrightarrow\oM_1(\PP^4,d)$ is a closed immersion. In particular $\Z$ is a proper DM stack.
\end{lem}
\begin{proof}
From the above description of arrows in $\Z$, $i$ is  representable (i.e. faithful)  and a monomorphism (i.e. full).

We can check properness using the valuative criterion. We argue as in \cite[Theorem 4.3]{RSPW}. Let $\dvr$ be a DVR scheme with generic point $\eta;$ consider a diagram:
\bcd
\cC_{\eta}\ar[r]\ar[d,"\phi_{\eta}" left]  & \cC_\dvr\ar[d,"\phi_\dvr" left]\ar[r,"f"] & \PP^4 \\
\hC_{\eta}\ar[r,"j" below]\ar[urr] & \hC_\dvr\ar[ur,dashed ,"g" below right] &
 \ecd 
Notice that there is an open dense substack of $\Z$ where $\phi$ is an isomorphism. Indeed the generic point of either the main component or any boundary component is already $1$-stable. Thus we can assume that $\phi_{\eta}$ in the above diagram is an isomorphism.

Observe that $f$ is constant on the fibers of $\phi_\dvr$, so it factors topologically through $\hC_{\dvr}.$
We can conclude as in \cite{RSPW} or appeal to \cite[Lemma~1.15]{debarre} using $\phi_*\OO_{\cC_\dvr}\cong\OO_{\hC_\dvr}$. To see this consider the exact sequence:
\[0\to \OO_{\hC_\dvr}\to\phi_*\OO_{\cC_\dvr}\to \phi_*\OO_{\cC_\dvr}/\OO_{\hC_\dvr}\to 0\]
Since $\phi$ is an isomorphism away from the cuspidal point, the cokernel is supported in dimension $0$. However $\chi (\OO_{\hC_\eta})=\chi(\phi_*\OO_{\cC_\eta})$ implies the same equality holds on the whole of $\dvr$, since the Euler characteristic is constant in flat families. So $\chi(\phi_*\OO_{\cC_\dvr}/\OO_{\hC_\dvr})=\operatorname{length}(\phi_*\OO_{\cC_\dvr}/\OO_{\hC_\dvr})=0.$
\end{proof}

So we may add the commutative diagram to the left:
\bcd
\oM_1(\PP^4,d)\ar[d] & \pazocal Z \ar[l,hook,swap,"i"]\ar[r]\ar[d]\ar[rd,phantom,"\Box"] & \overline{\mathcal M}_1^{(1)}(\PP^4,d) \ar[d] \\
\MM^{\rm{wt,st}}_1 & \X\ar[l,swap,"\sim"]\ar[r] & \MM_1^{\rm{wt,st}}(1)
\ecd

 A description of the irreducible components of $\Z$ can be easily obtained from the inclusion $i$: there is a main component $\Z^\text{main}$ which is the closure of the locus of maps from a smooth elliptic curve, and for every $k\geq 2$ a boundary component $D^k\Z$, whose general point represents a contracted elliptic curve with $k$ many rational tails of positive degree. 
\begin{remark}
Each component of $\Z$ is isomorphic to the corresponding one in $\oM_1(\PP^4,d)$. 
Indeed given any stable map there is at most one factorisation through the weighted $1$-stabilisation of the curve and more in details: objects of $\oM_1(\PP^4,d)\setminus D^1$ are $1$-stable already; objects of $D^1\cap \oM_1(\PP^4,d)^\text{main}$ factor through the cusp thanks to the smoothability criterion, which implies the node is a ramification point of the map; and objects of $D^1\cap D^k$ ($k\geq 2$) factor through a map which is constant on the cusp.

\begin{figure}
\centering
\begin{tikzpicture}
\draw (0,0) circle (.5cm);
\draw[fill=gray!50] (.7,.7) circle (.5cm);
\draw[fill=gray!50] (.7,-.7) circle (.5cm);
\draw (-1,0) ellipse (.5cm and 1cm);
\draw (-1.05,.7) .. controls (-.85,.5) and (-.85,-.5) .. (-1.05,-.7);
\draw (-.95,.5) .. controls (-1.2,.5) and (-1.2,-.5) .. (-.95,-.5);

\draw[->] (1.5,0) -- (2.5,0);
\draw (2.9,0) .. controls (3.2,0) and (3.1,.5) .. (3.5,.5);
\draw (3.5,0) ++(90:.5cm) arc (90:-90:.5cm);
\draw (2.9,0) .. controls (3.2,0) and (3.1,-.5) .. (3.5,-.5);
\draw[fill=gray!50] (4.2,.7) circle (.5cm);
\draw[fill=gray!50] (4.2,-.7) circle (.5cm);
\end{tikzpicture}
\caption{A typical element of $D^1\cap D^2$.}
\end{figure}
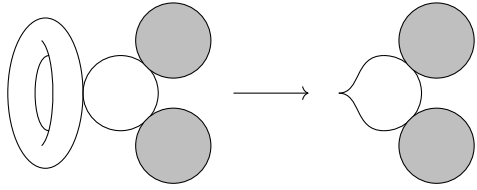

On the other hand, objects of $D^{1,\circ}=D^1\setminus(\oM_1(\PP^4,d)^\text{main}\cup\bigcup_{k\geq 2}D^k)$ do not admit any factorisation, so $\Z$ has no corresponding component.
\end{remark}
We introduce two more spaces: let $\XP$ and $\Z^p$ be the algebraic stacks defined by the following cartesian diagram:
\bcd
\Z^p \ar[r,"\alpha"]\ar[d,"\gamma"]\ar[rd,phantom,"\Box"] & \overline{\mathcal M}_1^{(1)}(\PP^4,d)^p \ar[d] \\
\XP \ar[r]\ar[d]\ar[rd,phantom,"\Box"] & \mathfrak{Pic}_1{(1)} \ar[d] \\
\X\ar[r] & \MM_1^{\rm{wt,st}}(1)
\ecd

\begin{remark}
The stack $\XP$ parametrises
\bcd
\cC_S \ar[rr,"\phi_S"]\ar[dr] & & \hC_S\ar[ld] \\
& S &
\ecd
with a line bundle $\hL_S$ on $\hC_S$. Notice that by pulling back $\hL_S$ via $\phi_S$ we obtain a line bundle on $\cC_S$, hence a morphism $\XP\to \mathfrak{Pic}_1$. This is generically an isomorphism, but has $1$-dimensional fibers over the locus of elliptic tails, due to the fact that $\Pic(\widehat C)\to\Pic(C)$ has kernel $\mathbb G_{\rm a}$ when $\widehat C$ has a cusp.
\end{remark}

\begin{remark}
Similarly the stack $\Z^p$ parametrises
\bcd
\cC_S\ar[rr,"\phi_S"]\ar[dr,"\pi_S" below left] & & \hC_S\ar[ld,"\hat\pi_S"] \ar[r,"f_S"] & \PP^4_S\\
& S & &
\ecd
with a $p$-field $\psi\in H^0(\hC_S,f_S^*\OO_{\PP^4}(-5)\otimes\omega_{\hat\pi_S})$.

We were not able to compare $\Z^p$ with $\oM_1(\PP^4,d)^p$ directly, since denoting by $\hL=f^*\OO_{\PP^4}(1)$ and by $\mathcal L=\phi^*\hL$ we only have a map $\R^1\hat\pi_*\hL\to\R^1\pi_*\mathcal L$ on $\Zp$ which is not an isomorphism (dually, since $\phi^*\omega_{\hat\pi}=\omega_{\pi}(\pazocal E)$, the $p$-fields on $\cC$ that we get by pulling back from $\hC$ vanish on elliptic tails).
\end{remark}

\subsection{A cosection-localised class on $\Zp$}\label{sec:cosecvirpull}

The obstruction theory $\R^\bullet\hat{\pi}_*(\hL^{\oplus 5}\oplus\widehat{\mathcal P})$ for the morphism $\overline{\mathcal M}_1^{(1)}(\PP^4,d)^p\to\mathfrak{Pic}_1{(1)}$ together with its cosection induce a localised virtual class on $\Z^p$ by localised virtual pullback, see \S~\ref{sec:cosection}:
\[[\Z^p]^{\rm{vir}}_{\alpha^*\sigma,\rm{loc}}:=\gamma^!_{\mathbb E_{\Z^p/\XP},\alpha^*\sigma}[\XP].\]

In order to show that this gives the same invariants as the space of $1$-stable maps with $p$-fields, we need to prove that localised virtual pullback commutes with proper pushforward.

\begin{lem}\label{lem:comutativity}
With the same setup and notation as in \S~\ref{sec:cosection}, suppose furthermore that $\varphi$ is proper and $\varphi'(Y')\cap U\neq\emptyset$. For every cycle $[B']\in A_*(S')$,
\[\varphi'_*\circ(\rho')^!_{\mathbb E_{Y'/S'}}[B']=(\rho)^!_{\mathbb E_{Y/S}}\circ\varphi_*[B]\in A_*(Y(\sigma)).\]

\end{lem}
\begin{proof}
Consider the following diagram:
\bcd
A_*(S')\ar[r]\ar[d] & A_*(\mathfrak C_{Y'/S'}) \ar[r]\ar[d] & A_*(h^1/h^0(\mathbb E_{Y'/S'})(\sigma'))\ar[r,"s_{\sigma',\rm{loc}}^!"]\ar[d] & A_*(Y'(\sigma'))\ar[d] \\
A_*(S)\ar[r] & A_*(\mathfrak C_{Y/S}) \ar[r] & A_*(h^1/h^0(\mathbb E_{Y/S})(\sigma))\ar[r,"s_{\sigma,\rm{loc}}^!"] & A_*(Y(\sigma))
\ecd
Only the right-most square is not classical. Hence we are left with showing that for every cycle $[B']\in A_*(h^1/h^0(\mathbb{E}_{Y'/S'})(\sigma'))$ we have:
\begin{equation}\label{eq:commcos}\varphi'_*( s^!_{\sigma',\rm{loc}}[B'])=s^!_{\sigma,\rm{loc}}\left(\varphi^{\mathbb{E}}_*[B']\right)\,\in\, A_*(Y(\sigma)),
\end{equation}
where $\varphi^{\mathbb{E}}\colon h^1/h^0(\mathbb{E}_{Y'/S'})(\sigma')\to h^1/h^0(\mathbb{E}_{Y/S})(\sigma)$ is induced by the morphism of vector bundle stacks.

If $B'\subset h^1/h^0(\mathbb{E}_{Y'/S'})\rvert_{Y'(\sigma')}$ then $\varphi^{\mathbb E}(B')\subset h^1/h^0(\mathbb{E}_{Y/S})\rvert_{Y(\sigma)}$,
thus the localised Gysin morphisms coincide with $s^!_{h^1/h^0(\mathbb{E}_{Y'/S'})\rvert_{Y'(\sigma')}}$ and $s^!_{h^1/h^0(\mathbb{E}_{Y/S})\rvert_{Y(\sigma)}}$, and their commutativity with proper pushforward is well-known, see \cite{Kresch} (or \cite[Chapter~3]{FUL} if we assume the existence of global resolutions).

Assume then that $B'$ is a prime cycle in $h^1/h^0(\mathbb{E}_{Y'/S'})(\sigma')$ not contained in $h^1/h^0(\mathbb{E}_{Y'/S'})\rvert_{Y'(\sigma')}$. Then recall that: 
 \[s^!_{\sigma',\rm{loc}}[B']:=\frac{1}{k} \nu'(\sigma')_*\left([D]\cdot s^!_{\widetilde{G}}[\widetilde{B}]\right)\,\in\,A_*(Y'(\sigma')),\]
 where $\nu'\colon \widetilde{Y}\to Y'$, $(\nu')^*(\sigma')$ and $\widetilde B$ satisfy the conditions of a regularising morphism, see \S~\ref{cond:deflocal}. Notice that if we consider the composition
 $\nu\colon \widetilde{Y}\xrightarrow{\nu'} Y'\xrightarrow{\varphi'} Y$ and the pull-back $\nu^*(\sigma)$, the first two of those conditions are always satisfied. How about
 $\widetilde{B}\to B'\to \varphi^{\mathbb{E}}(B')=:B?$
 Suppose first that $B'\to B$ is generically finite of degree $h$; then $\widetilde{B}\to B$ is  generically finite of degree $hk$, so:
 \begin{align*}
 \varphi'_*( s^!_{\sigma',\rm{loc}}[B'])&= \frac{1}{k}\varphi'_*\circ \nu'(\sigma')_*\left([D]\cdot s^!_{\widetilde{G}}[\widetilde{B}]\right)\\
 &=\frac{1}{k} \nu(\sigma)_*\left([D]\cdot s^!_{\widetilde{G}}[\widetilde{B}]\right)\\
 &=h\, s_{\sigma,\rm{loc}}^![B]\\
 &=s_{\sigma,\rm{loc}}^!(\varphi^{\mathbb{E}}_*[B']).
 \end{align*}
If instead $B'\to B$ is not generically finite, then the right-hand side in formula \eqref{eq:commcos} is zero. In order to show that the same is true for the left-hand side it is enough to appeal to \cite[Lemma~2.7]{KLcosection}, since $\nu_*^{\mathbb{E}}[\widetilde B]=(\nu')_*^{\mathbb E}\circ \phi^{\mathbb E}_*[\widetilde B]=0.$ 
\end{proof}

\begin{cor}\label{cor:Zp-vs-pfields}
$\deg[\Z^p]^{\rm{vir}}_{\rm{loc}}=\deg[\Mone{1}{\PP^4}{d}^p]^{\rm{vir}}_{\rm{loc}}$
\end{cor}

This descends from the following lemma and the fact that $\XP$ and $\mathfrak{Pic}_1(1)$ are birational.

\begin{lem}
 The $1$-stabilisation $\mathfrak{M}_1^{\rm{wt,st}}\to \mathfrak{M}_1^{\rm{wt,st}}(1)$ is a proper morphism.
\end{lem}
\begin{proof}
Use the valuative criterion: let $\dvr$ be a DVR scheme with generic point $\eta$; we have to fill in the upper right part of the following diagram:
\bcd
\mathcal C_{\eta}\ar[d,"\phi_{\eta}"] \ar[r,dashed, hookrightarrow] &\mathcal C\ar[d, dashed, "\phi"]\\
\hC_{\eta}\ar[r, hookrightarrow] &\hC
\ecd
We may in fact assume that $\phi_{\eta}$ is an isomorphism, and then it is enough to take the weighted stable model of $\hC.$
\end{proof}

\section{Local equations and desingularisation}\label{sec:equations}
\subsection{Equations for $\Z^p$ relative to $\XP$}
We are going to need a description of the normal cone $\mathfrak C_{\Z^p/\XP}$ in order to perform a splitting.
Since $\Z^p$ is a line bundle over the boundary of $\Z$, we find instead equations for the latter.

Recall that $\Z$ can be embedded as an open inside $C(\hat{\pi}_*\hL^{\oplus 5})$ over $\XP$. We are going to find an embedding of $C(\hat{\pi}_*\hL^{\oplus 5})$ in a smooth ambient space over $\XP$, that will be a vector bundle obtained by suitably twisting $\hL$.

Following \cite{HL}, we work locally on $\Z$: start with a point $\xi=[(C\xrightarrow{\phi}\widehat{C}\xrightarrow{f}\PP^4]\in\Z$ and choose coordinates on $\PP^4$ such that $f^{-1}\{x_0=0\}$ is a simple smooth divisor $D=\sum_{i=1}^d\delta_i$ on $\widehat{C}$. This continues to be true on a neighbourhood $\pazocal U$ of $\xi$.

 Locally the morphism $\Z\to\XP$ can be written as $\xi\mapsto[C\to\widehat{C},\OO_{\widehat{C}}(D)]$, which admits a local lifting $\pazocal U\to\X^{\rm{div}}:=\MM_1^{\rm{div}}(1)\times_{\hM}\X$, and in fact hits the smooth locus of the latter.
 
 \begin{remark}
  The projection $\X^{\rm{div}}\to \XP$ 
  is \emph{not} smooth. In fact, when the line bundle is trivial on the minimal elliptic subcurve $E$, it may be deformed to a degree $0$, non-effective line bundle on such a subcurve, so that sections of $\OO_{\widehat{C}}(D)$ which are constant and non-zero on $E$ are obstructed.

There is a way around this: in a neighbourhood $\pazocal V\subseteq \XP$ of $[C\to\widehat{C},\OO_{\widehat{C}}(D)]$ we can write the universal line bundle $\hL_{\pazocal V}$ as $\OO_{\hC_{\pazocal V}}(\pazocal D + p-p_0)$. Indeed we can pick a local section $p_0$ through the core, so that $\hL_{\pazocal V}(p_0)$ becomes effective. We should think of $p$ as a local coordinate on $\XP$ relative to $\X$.
\end{remark}

Locally on $\pazocal V$ we can pick another smooth section $\A$ of the core not intersecting $p_0$, neither the support of $\pazocal D+p$.

\begin{lem}
$C(\hat\pi_*\hL_{\pazocal V})$ is the kernel of the vector bundle map:
\[ \hat\pi_*\OO_{\hC}(\A+\pazocal D+p-p_0)\xrightarrow{\varphi}\hat\pi_*\OO_{\A}(\A)\]
Up to shrinking $\pazocal V$ we may write:
\[\hat\pi_*\OO_{\hC}(\A+\pazocal D+p-p_0)\cong\bigoplus_{i=1}^d\hat\pi_*\OO_{\hC}(\A+\pazocal \delta_i-p_0)\oplus\hat\pi_*\OO_{\hC}(\A+p-p_0)\]
\end{lem}
Compare with \cite[Lemma 4.10]{HL}. Denote by \[\varphi_i\colon\hat\pi_*\OO_{\hC}(\A+\pazocal \delta_i-p_0)\to\hat\pi_*\OO_{\A}(\A)\] (and similarly $\varphi_p$) the composite of the inclusion with $\varphi$.

Let us introduce some more notation: around a point $[\widehat{C}]\in\MM_1^{\rm{wt,st}}(1)$, for every node $q$ of $\widehat{C}$ there is a coordinate $\zeta_q$ whose vanishing locus is the divisor where such a node is \emph{not} smoothed. These functions can be pulled back to $\pazocal V$. Denote by \[\zeta_{[\delta_i,\pazocal A]}=\prod\zeta_q\]
where the product runs over all the nodes separating $\delta_i$ from the core.
 
\begin{lem}
There are trivialisations for the line bundles $\pi_*\OO_{\hC}(\A+\pazocal \delta_i-p_0),$ $\hat\pi_*\OO_{\hC}(\A+p-p_0),\, \pi_*\OO_{\A}(\A) $ such that locally we have the following explicit expression for $\varphi_i$ and $\varphi_p$:
\[\varphi_i=\zeta_{[\delta_i,\pazocal A]}, \quad \varphi_p=(p-p_0)\]
\end{lem}
Compare with \cite[Proposition 4.13]{HL}.
\begin{remark}
The vanishing locus of $(p-p_0)$ on the boundary means that the line bundle restricts to the trivial one on the core.
\end{remark}

\begin{lem}\label{lem:equations}
A local chart $\pazocal U$ for $\Z$ can be embedded as an open inside:
\[ (F_0=\ldots=F_4=0)\subseteq \rm{Vb}(\hat{\pi}_*\hL_{\pazocal V}(\A)^{\oplus 5}) \]
where
\[ F_j=\sum_{i=1}^d \zeta_{[\delta_i,\pazocal A]}w_i^j+(p-p_0)w_{d+1}^j \]
and $w_i^j$ are coordinates on the fiber of the $j$-th copy of $\rm{Vb}(\hat{\pi}_*\hL_{\pazocal V}(\A))$ over $\pazocal V$.
\end{lem}
Compare with \cite[Theorems 2.17-19]{HL}.

\subsection{Hu-Li blow-up and desingularisation}
We perform a modular blow-up of $\MM_1^{\rm{wt,st}}(1)$: we successively blow up $\widehat{\Theta}_k, \;k\geq 2,$  defined as the closure of the loci where the minimal elliptic subcurve $E$ has weight $0$ and $\left\lvert\overline{C\setminus E}\cap E\right\rvert=k.$

Notice that after the $k$-th blow-up, the strict transform of $\widehat{\Theta}_{k+1}$ is smooth, so the final result  $\widetilde{\MM}_1^{\rm{wt,st}}(1)$ is smooth as well. 

\begin{remark}
The fiber product 
\[\widetilde{\MM}_1^{\rm{wt,st}}(1)\times_{\MM_1^{\rm{wt,st}}(1)}\MM_1^{\rm{wt,st}}\]
recovers the Hu-Li blow-up $\widetilde{\MM}_1^{\rm{wt,st}}.$ The key observation is that $\Theta_1$ is already a Cartier divisor and the inverse image of $\widehat{\Theta}_k$ is precisely $\Theta_k$; 
using the universal property of the blow-up, it can be shown that there are maps in both directions, and they are inverse to one another. 
\end{remark}

\begin{remark}\label{rmk:eqn}
After blowing up, the equations in \ref{lem:equations} are simplified and assume the following form:
\[\tilde\zeta\tilde w+(p-p_0)w_{d+1}=0\]
where $\tilde\zeta$ is one of the newly created boundary divisors $\widetilde{\Theta}_k$ from $\widetilde{\MM}_1^{\rm{wt,st}}$ (i.e. one of the exceptional divisors produced by the blow-up process), and $\tilde w$ is a suitably defined coordinate on the fiber of $\rm{Vb}(\hat{\pi}_*\hL(\A))\times_{\XP}\widetilde{\XP}$.
\end{remark}

Summing up, we get:
\bcd
   &  \widetilde{\Z}^p \ar[r]\ar[d]\ar[rd,phantom,"\Box"] & \tM_1^{(1)}(\PP^4,d)^p \ar[d] \\
\tM_1(\PP^4,d)\ar[d] & \widetilde{\Z} \ar[l,hook,swap,"i"]\ar[r]\ar[d]\ar[rd,phantom,"\Box"] & \tM_1^{(1)}(\PP^4,d) \ar[d] \\
\widetilde{\mathfrak{Pic}}_1\ar[d] &\widetilde{\XP}\ar[l]\ar[r]\ar[d]\ar[rd,phantom,"\Box"] & \widetilde{\mathfrak{Pic}}_1 (1)\ar[d]\\
\widetilde{\MM}^{\rm{wt,st}}_1 & \widetilde{\X}\ar[l,swap,"\sim"]\ar[r] & \widetilde{\MM}_1^{\rm{wt,st}}(1)
\ecd

Notice that the components of $\tZp$ are in bijection with those of $\Zp$, however all the boundary ones have the same dimension $5d+4$, and their intersection with main is a divisor in the latter.

We conclude this brief section remarking that the blow-up procedure does not affect the invariants:
\begin{lem}\label{lem:tilding}
We have the identity:
\[\deg\virloc{\tZp}=\deg\virloc{\Zp}.\]
\end{lem}
Compare with \cite[Proposition~2.5]{CL}.
\section{Splitting the cone and proof of the Main Theorem}\label{section:main_proof}
We are finally able to study the cone $\mathfrak{C}_{\tZp/\tXP}.$ This is going to be a word-by-word repetition of the arguments in \cite{CLpfields}.
\begin{lem}\label{lem:open_cones}
The map $\tilde{\rho}\colon \tZp\to\tXP$ has a relative perfect obstruction theory $\mathbb {E}_{\tZp/\tXP}=\R^{\bullet}\hat\pi_*(\hL^{\oplus 5}\oplus\widehat{\mathcal P})$. The intrinsic normal cone $\mathfrak{C}_{\tZp/\tXP}$ has the following properties:
\begin{enumerate}
\item Its restriction to the open $\tZ^{p,\circ}= \tZ^{p,\rm{main}}\setminus \bigcup_{k\geq 2}D^k \tZp$ can be described as the zero section of $h^1/h^0(\mathbb {E}_{\tZp/\tXP})\rvert_{\tZ^{p,\circ}}$
\item Its restriction to the open $\tZ^{p,\rm{gst},\circ}=\tZp - \tZ^{p,\rm{main}}$  is a rank $2$ subbundle stack of $h^1/h^0(\mathbb {E}_{\tZp/\tXP})\rvert_{\tZ^{p,\rm{gst},\circ}}$
\end{enumerate}
\end{lem}
\begin{proof} Compare with \cite[Lemma 4.3]{CLpfields}
\begin{enumerate}
\item Observe that $\tZ^{p,\circ}\cong \tZ^{\circ}$ because here $H^0(\widehat{C}, \widehat{L}^{\otimes -5}\otimes \omega_{\widehat{C}})=0$. Moreover $\tZ^{\circ}$ is unobstructed on its image, which is an open of $\tXP$, because $\R^1\hat{\pi}_*\hL=0$.
So the normal cone is $[\tZ^{p,\circ}/\hat{\pi}_*\hL^{\oplus 5}]$, which is the zero section of $h^1/h^0(\mathbb {E}_{\tZp/\tXP})\rvert_{\tZ^{p,\circ}}=[0\oplus \R^1\hat{\pi}_*\widehat{\mathcal P}/\hat{\pi}_*\hL^{\oplus 5}\oplus 0]$.

\item We know that $\tZ^{p,\rm{gst},\circ}$ is a line bundle over $\tZ^{\rm{gst},\circ}$. From the equations \ref{rmk:eqn} we see that the latter is smooth over its image $\pazocal W$ in $\tXP$, which is the codimension $2$ locus where the core has weight $0$ and the line bundle is trivial on it. Recall that every smooth morphism $A\to B$ of relative dimension $n$ factors as $A\xrightarrow{\acute{e}t} B\times \Aaff^n\xrightarrow{\pr_1} B$. So we have
\bcd
\tZ^{p,\rm{gst},\circ} \ar[r,"\acute{e}t"] & \pazocal W\times\Aaff^{5d+6}\ar[r,hook]\ar[d,"q"] & \tXP\times \Aaff^{5d+6}\ar[d] \\
& \pazocal W \ar[r,hook] & \tXP
\ecd
where the bottom horizontal arrow is a codimension $2$ regular embedding. Thus 
\[\mathfrak{C}_{\tZp/\tXP}\rvert_{\tZ^{p,\rm{gst},\circ}}\cong \left[ q^* C_{\pazocal W/\tXP}/ \hat{\pi}_*\hL^{\oplus 5}\oplus\hat{\pi}_*\widehat{\mathcal P}\right]\]
is a rank 2 subbundle stack of  $h^1/h^0(\mathbb {E}_{\tZp/\tXP})\rvert_{\tZ^{p,\rm{gst},\circ}}.$
\end{enumerate}
\end{proof}
Notice that the image of $\tZ^{\circ}$ in $\tM_1(\PP^4)$ contains $\tM_1(\PP^4)^{\rm{main}}\cap \widetilde{D}^1.$

\subsection{Contribution of the main component}
Recall the definition of the \emph{closure of the zero section of a vector bundle stack}: let $B$ be an integral algebraic stack and let $\mathbb F=[F_0\xrightarrow{d} F_1]$ be a complex of locally free sheaves on $B$. The zero section is $0_{\mathbb F}\colon [F_0/F_0]\to h^1/h^0(\mathbb F)=[F_1/F_0]$, which is in general not a closed embedding; its closure is defined as:
\[ \overline{0}_{\mathbb F}=[\operatorname{cl}(dF_0)/F_0] \]
$\overline{0}_{\mathbb F}$ is an integral stack.

\begin{ex}
When $h^0(\mathbb F)=0$, the closure of the zero section looks like $B$ with some further stacky structure on the vanishing locus of $d$. Consider for example $B=\PP^1$ and $\mathbb F=[\OO_{\PP^1}\xrightarrow{x}\OO_{\PP^1}(1)]$. Then the action of $e\in F_0$ on $F_1$ is given by $f\mapsto f+xe$. Clearly $\operatorname{cl}(dF_0)$ is the whole line bundle $F_1$; the $F_0$-action is transitive on the fibers over $\{x\neq 0\}$ and trivial on the $\{x=0\}$-fiber. Hence $\overline{0}_{\mathbb F}$ is isomorphic to $\PP^1\setminus\{x=0\}$ with a gerbe $[\Aaff^1/\mathbb G_{\rm a}]$ replacing the point $\{x=0\}$.
\end{ex}

We may now split the cone $\mathfrak{C}_{\tZp/\tXP}$ in the following manner: we denote by $\mathfrak C^{\rm{main}}$ the closure of the zero section of $h^1/h^0(\mathbb {E}_{\tZp/\tXP})\rvert_{\tZ^{p,\rm{main}}}$, which is an irreducible cone supported on the main component. All the rest is supported on the boundary components, possibly on their intersection with the main one, and we are going to pack all the components supported on $D^k\tZp$ together and label them $\mathfrak C^k$ accordingly, so in the end we obtain a splitting:
\[
 \mathfrak{C}_{\tZp/\tXP}=\mathfrak C^{\rm{main}}+\sum_{k\geq 2} \mathfrak C^k
\]
We are going to show that:
\begin{enumerate}
 \item the contribution of $\mathfrak C^{\rm{main}}$ is exactly the reduced invariants of $X$;
 \item the other cones $\mathfrak C^k,\ k\geq 2,$ are enumeratively meaningless.
\end{enumerate}

In order to prove the first claim we proceed as in \cite[\S5]{CLpfields}; let us start by noticing that the obstruction theory $\mathbb E:=\mathbb E_{\tZp/\tXP}$ splits as $\mathbb E_1\oplus\mathbb E_2$ where $\mathbb E_1=\R^{\bullet}\hat\pi_*(\hL^{\oplus 5})$ and $\mathbb E_2=\R^{\bullet}\hat\pi_*(\widehat{\mathcal P})$. When we restrict to $\tZ^{p,\rm{main}}$ we see that $h^1/h^0(\mathbb E_1)$ is the closure of its own zero section; it follows that:
\[
 \mathfrak C^{\rm{main}}=\overline{0}_{\mathbb E\rvert_{\tZ^{p,\rm{main}}}}=h^1/h^0(\mathbb E_1)\rvert_{\tZ^{p,\rm{main}}}\oplus \overline{0}_{\mathbb E_2\rvert_{\tZ^{p,\rm{main}}}}
\]
Then by standard intersection theory (pullback the right-hand side by $\hat\pi_{\mathbb E}^*=\hat\pi_{\mathbb E_1}^*\circ\hat\pi_{\mathbb E_2}^*$):
\[
 0^!_{\mathbb E}[\mathfrak C^{\rm{main}}]=0_{\mathbb E_2}^![\overline{0}_{\mathbb E_2\rvert_{\tZ^{p,\rm{main}}}}]
\]

At this point we recall the following \cite[Lemma 5.3]{CLpfields}:
\begin{lem}
Let $\mathbb E=[E_0\to E_1]$ be a complex of locally free sheaves on an integral Deligne-Mumford stack $B$ such that
$h^1(\mathbb E)$ is a torsion sheaf on $B$ and the image sheaf of $E_0\to E_1$ is locally free.
Let $U\subseteq B$ be the complement of the support of $h^1(\mathbb E)$, and let $\mathbf{B}\subseteq h^1/h^0(\mathbb E^\vee[-1])$
be the closure of the zero section  of the vector bundle $h^1/h^0(\mathbb E^\vee[-1]|_U)= h^0(\mathbb E|_U)^\vee$. Then
$$0^![\mathbf{B}]=e(h^0(\mathbb E)^\vee)\in A_*(B).
$$
\end{lem}

We apply this lemma to $\R^{\bullet}\hat{\pi}_*\hL^{\otimes 5}=\mathbb E_2^\vee$ on $\tZ^{p,\rm{main}}$. Notice that it satisfies the hypotheses by virtue of the equations in Remark \ref{rmk:eqn}: indeed the question may be addressed locally; looking at the resolution of $\mathbb E_2^\vee$:
\[
 \hat\pi_*\hL^{\otimes 5}(\A)\to \hat\pi_*\hL^{\otimes 5}(\A)\rvert_{\A}
\]
we deduce from the equation that the image of this map is $\pi_*\hL^{\otimes 5}(\A)\rvert_{\A}\otimes\OO_{\tZ^{p,\rm{main}}}(-\Xi)$, where $\Xi$ denotes the Cartier divisor $\Xi=\tZ^{p,\rm{main}}\cap\left(\bigcup_{k\geq 2}D^k\tZp\right)$. Then $\hat\pi_*\hL^{\otimes 5}$ is a vector bundle, being the kernel of a vector bundle map.

\begin{lem}\label{lem:main-compo}
 If we let $i$ be the inclusion of $\tZ$ in $\tM_1(\PP^4,d)$, then:
 \[
  i_*(c_{\rm{top}}(\hat{\pi}_*\hL^{\otimes 5})\cap[\tZ^{p,\rm{main}}])=c_{\rm{top}}({\pi}_*\mathcal L^{\otimes 5})\cap[\tM_1(\PP^4,d)^{p,\rm{main}}]
 \]
\end{lem}
\begin{proof}
 Recall that the projection $(-)^p\to (-)$ is an isomorphism on main, so it makes sense to write $i_*[\tZ^{p,\rm{main}}]=[\tM_1(\PP^4,d)^{p,\rm{main}}]$, which follows from Lemma~\ref{lem:inclusion}. On the other hand notice that on $\tZ^{p,\rm{main}}$ we have:
 \[
  \pi_*\mathcal L=\hat{\pi}_*\phi_*\phi^*\hL=\hat{\pi}_*\hL
 \]
by projection formula and since $\phi_*\OO_{\cC_{\tZ^{p,\rm{main}}}}=\OO_{\hC_{\tZ^{p,\rm{main}}}}$ by Zariski connectedness theorem. The result follows from the projection formula for Chern classes.
\end{proof}

We are left with showing that the rest of the $\mathfrak C^k$ do not contribute to the invariants. This is essentially a dimensional computation. The arguments of \cite[\S\S6-8]{CLpfields} may be adapted; we shall outline them for the benefit of the reader. We introduce the notation $\tZ^{p,\rm{gst}}:=\bigcup_{k\geq 2} D^k\tZp$ for the union of the boundary components, and $\mathfrak C^{\rm{gst}}=\bigcup_{k\geq 2}\mathfrak C^k$.

\subsection*{Step I: reduction to the case of a cone inside a vector bundle}

This section deals with removing the technicalities of working with a cone stack inside a vector bundle stack.

The key point is that $\mathbb E:=\mathbb E\rvert_{\tZ^{p,\rm{gst}}}$ has locally free $h^0$ \emph{and} $h^1$: the equations in Remark \ref{rmk:eqn} are automatically satisfied on the boundary, without imposing any condition on the fiber coordinates.

When we fix a resolution by locally free sheaves $\mathbb E=[F_0\xrightarrow{d} F_1]$, the image sheaf $d(F_0)$, which is the kernel of $F_1\to h^1(\mathbb E)$, is a subbundle of $F_1$. Consider the projections:
\[
 \beta\colon F_1\to h^1/h^0(\mathbb E) \quad \text{and} \quad \beta'\colon F_1\to \tilde{V}:=\R^1\hat{\pi}_*(\hL^{\oplus 5}\oplus\widehat{\mathcal P});
\]
the second is also flat since $d(F_0)$ is a vector bundle. The cone stack $\mathfrak C^{\rm{gst}}$ can be descended to a cone $C^{\rm{gst}}\subseteq \tilde{V}$ by taking the quotient of $\beta^{-1}\mathfrak C^{\rm{gst}}$ by the free action of $d(F_0)$; $C^{\rm{gst}}$ should then be thought of as the coarse moduli of $\mathfrak C^{\rm{gst}}$. Recall that we started with a cosection $\sigma_1$ of $V$ (see Eq. \eqref{eqn:cosection}) and pulled it back to all relevant spaces. 
It follows from the commutativity of localised Gysin pullback with flat pullback that:
\[
 s^!_{h^1/h^0(\mathbb E),\tilde\sigma_1}[\mathfrak C^{\rm{gst}}]= s^!_{\tilde V,\tilde\sigma_1}[C^{\rm{gst}}]
\]
See \cite[Proposition 6.3]{CLpfields}.

\subsection*{Step II: from cosection localised to standard Gysin map}
Recall that the construction of a localised virtual class refines the standard one, namely:
\[\iota_*\virloc{Y}=\vir{Y},\]
where $\iota\colon Y(\sigma)\hookrightarrow Y$ is the degeneray locus of the cosection. In particular, when $Y$ itself is a proper Deligne-Mumford stack, the degree of the localised virtual class can be computed after pushing it forward to $Y$.

Then to compute $s^!_{\tilde V,\tilde\sigma_1}[C^{\rm{gst}}]$ we are going to compactify $\tZ^{p,\rm{gst}},$ extend the cone, the vector bundle and the cosection to it, and then make use of the fact we just recalled.

Since $\tZ^{p,\rm{gst}}\cong\operatorname{Vb}_{\tZ^{\rm{gst}}}\left(\hat{\pi}_*\widehat{\mathcal P}\right)$, we take its standard copactification:
\[\bar{\gamma}\colon \overline{\Z}^{p,\rm{gst}}:=\PP\left(\hat{\pi}_*\widehat{\mathcal P}\oplus \OO_{\tZ^{\rm{gst}}}\right)\to \tZ^{\rm{gst}}.\]
We want to pullback the cosection:
\[\sigma_{1|(u,\psi)}(\mathring{x},\mathring{p})=\mathring{p}\w(u)+\psi\sum_{i=0}^4\partial_i\w(u)\mathring{x}_i;\]
noticing that in the compactification the $p$-field $\psi$ can go to infinity, we define the vector bundles on $\oZp$:
\[\tilde{V}^{\rm{cpt}}_1=\bar{\gamma}^*V_1(-D_{\infty}),\qquad \tilde{V}^{\rm{cpt}}_2=\bar{\gamma}^*V_2.\]
We are now able to extend the cosection and get:
\[\bar{\sigma}\colon \tilde{V}^{\rm{cpt}}=\tilde{V}^{\rm{cpt}}_1\oplus \tilde{V}^{\rm{cpt}}_2\to\OO_{\oZp}. \]
 
 The compatibility of $\tilde\sigma$ with $\bar\sigma$, and the fact that $\oZp$ is proper, together with the observation at the beginning of this section, explain the following:
 \begin{prop}
 Let $\iota_{!}\colon Z_*(\tilde{V}(\widetilde{\sigma}))\to Z_*(\tilde{V}^{\rm{cpt}})$ be defined by 
 $\iota_{!}[C]=[\overline{C}].$ And $i\colon D(\widetilde{\sigma})\to \tZ^{\rm{gst}}$ the inclusion. Then
 \[\bar{\gamma}_*\circ s^{!}_{\tilde{V}^{\rm{cpt}}}\circ\iota_{!}=i_*\circ s^{!}_{\tilde\sigma,\rm{loc}}\colon  Z_*(\tilde{V}(\widetilde{\sigma}))\to A_*(\tilde{Z}^{\rm{gst}}).\]
 \end{prop}
 See \cite[Proposition~6.4]{CL} for full details.
 
 Furthermore, from functoriality of Gysin pullbacks and the deformation to the normal cone it follows that:
 \[s^!_{\tilde V^{\rm{cpt}}}[\overline C]=s^!_{\tilde V_2^{\rm{cpt}}}\circ s^!_{\tilde V_1^{\rm{cpt}}}[N_{\overline C\cap 0\oplus\tilde V_2^{\rm{cpt}}}\overline C].\]
 
 \subsection*{Step III: homogeneous cones}
Chang an Li introduce the notion of \emph{homogeneity} for substacks of $\tilde V$ on $\tZ^{p,\rm{gst}}$: write 
\[\tilde V=\tilde V_1\oplus\tilde V_2\;\;\text{ with} \;\;\tilde V_1=\R^1\hat{\pi}_*(\hL^{\oplus 5})\;\; \text{and}\;\; \tilde V_2=\R^1\hat{\pi}_*(\widehat{\mathcal P}),\]
 and $\gamma\colon \tZ^{p,\rm{gst}}\to\tZ^{\rm{gst}}$ the projection. Since $\tZ^{p,\rm{gst}}$ is the total space of a line bundle over $\tZ^{\rm{gst}}$, it comes with a natural $\Gm$-action on the fibers of $\gamma$. Moreover, since the highest pushforwards satisfy cohomology and base-change:
 \[\tilde V_i=\gamma^*V_i\]
 for the corresponding vector bundles $V_i$ on $\tZ^{\rm{gst}}$, so the total space of $\tilde V_i$ can be endowed with a $\Gm$-action that makes the projection to $\tZ^{p,\rm{gst}}$ equivariant. We say that a closed substack of $\tilde V$ is $0$-homogeneous if it is the pullback of a closed substack of $V$ along $\gamma$. In fact there are different $\Gm$-actions on $\tilde V$ that make the projection to $\tZ^{p,\rm{gst}}$ equivariant: namely, we can twist the trivial action on the fibers by two characters of $\Gm$, one for each $\tilde V_i$. Then we say that a substack of $\tilde V$ is $(l_1,l_2)$-homogeneous if it is invariant with respect to such an action. Here is how we are going to use the homogeneity:

\begin{lem}
 Let $C\subseteq \tilde V$ be an $(l_1,l_2)$-homogeneous sub\emph{cone} of $\tilde V$; then the cone $\overline{C}\cap(0\oplus\tilde V_2^{\rm{cpt}})$ is pulled back from a cone in $V_2$ on $\tZ^{\rm{gst}}$.
\end{lem}
\begin{proof}
 Locally we may pick coordinates $t$ on the fibers of $\gamma$, $x_1,\ldots,x_5$ on the fibers of $\tilde V_1^{\rm{cpt}}$, and $y_1,\ldots,y_{5d+5}$ on the fibers of $\tilde V_2^{\rm{cpt}}$, such that the ideal of $\overline{C}$ is generated by separately homogeneous polynomials $p_j$ in $t^{-l_1}x_i$ and $t^{-l_2}y_i$. The ideal of $\overline{C}\cap(0\oplus\tilde V_2^{\rm{cpt}})$ is then given by $\langle x_1,\ldots, x_5,p_j(0,t^{-l_2}y)\rangle_j$, where $p_j(0,t^{-l_2}y)$ results from setting $x_i=0$ in $p_j$. Notice now that $\overline{C}$ being a cone, it is invariant by scalar multiplication on the fibers of $\tilde V^{\rm{cpt}}$, so we may as well say that $\overline{C}\cap(0\oplus\tilde V_2^{\rm{cpt}})$ is cut by the ideal $\langle x_1,\ldots, x_5,p_j(0,y)\rangle_j$. This makes it clear that $\overline{C}\cap(0\oplus\tilde V_2^{\rm{cpt}})$ is pulled back from $(0\oplus V_2)$ on $\tZ^{\rm{gst}}$.
\end{proof}

Finally Chang and Li point out that the coarse moduli cone $C^{\rm{gst}}$ is $(0,1)$-homogeneous \cite[Proposition 6.7]{CLpfields}.

\subsection*{Step IV: reduction of the support of the cone}
We now explain a key technical lemma which will enable us to show that $C^{\rm{gst}}$ pushes forward to zero under a suitably defined morphism. It is basically reducing the support of $\overline{C^{\rm{gst}}}\cap 0\oplus\tilde{V}_2^{\rm{cpt}}$ to a manageable substack of $\tilde{V}_2^{\rm{cpt}}$, that is the union of the zero-section (i.e. $\oZp$) and a line subbundle of $\tilde{V}_2^{\rm{cpt}}$ supported on $\widetilde{\Xi}=\tZ^{p,\rm{main}}\cap\tZ^{p,\rm{gst}}$. Even better, using the homogeneity we can show that such a line bundle comes from $\Xi=\tZ^{\rm{main}}\cap\tZ^{\rm{gst}}$. We give a sketch of the proof, see \cite[Proposition 7.1]{CLpfields} for more details.
\begin{lem}
 There is a line subbundle $F$ of $V_2\rvert_\Xi$ such that:
 \[
  \overline{C^{\rm{gst}}}\cap 0\oplus\tilde{V}_2^{\rm{cpt}}\subseteq 0_{\tilde{V}_2^{\rm{cpt}}}\cup \tilde F:= \oZp\cup \bar\gamma^* F
 \]
\end{lem}
It is enough to show this before taking the closure. First they use the fact that there is a triple of compatible obstruction theories for the triangle:
\bcd
\tZp \ar[rr,"\gamma"]\ar[dr] & & \tZ \ar[dl] \\
& \widetilde{\XP} &
\ecd
such that their restrictions to $\tZ^{p,\rm{gst}}$ have locally free $h^0$ and $h^1$. By taking $h^1$ of the dual obstruction theories we obtain a commutative diagram:
\bcd
h^1(\mathbb L_{\tZp/\tZ}^\vee\rvert_{\tZ^{p,\rm{gst}}}) \ar[r]\ar[d] & h^1(\mathbb L_{\tZp/\XP}^\vee\rvert_{\tZ^{p,\rm{gst}}})\ar[r]\ar[d] &h^1(\gamma^*\mathbb L_{\tZ/\XP}^\vee\rvert_{\tZ^{p,\rm{gst}}})\ar[d] \\
\tilde{V}_2\ar[r,"i_2"] & \tilde{V}_1\oplus\tilde{V}_2 \ar[r,"\pr_1"] & \tilde{V}_1
\ecd
The vertical arrows are injective by the definition of an obstruction theory, and the bottom triangle is exact. Notice that $0\oplus\tilde{V}_2$ is precisely the kernel of $\pr_1$. It follows that, in order to understand the support of $C^{\rm{gst}}\cap 0\oplus\tilde{V}_2$, it is enough to study that of $N$, where $N$ is the coarse moduli cone of $\mathfrak C_{\tZp/\tZ}$, living in the upper left corner of the above diagram.

This is an easier task, since we know that $\tZp/\tZ$ is a line bundle on $\tZ^{\rm{gst}}$ and an isomorphism on $\tZ^{\rm{main},\circ}$. Hence we can always find a local chart $S\to \tZ$ and a diagram as follows:
\bcd
\tZp\ar[d] & T\ar[l]\ar[ld,phantom,"\Box"]\ar[d]\ar[r,hook,"V(\tilde\zeta t)"] & S\times\Aaff^1_t\ar[ld] \\
\tZ & S\ar[l,"\acute{e}t"] &
\ecd
where $\tilde\zeta$ is a local equation for the boundary. Then $\tau^{\geq-1}\mathbb L_{\tZp/\tZ}\rvert_T=[I/I^2\xrightarrow{\delta}\Omega_{\Aaff^1_S/S}]$; $I$ is generated by $\tilde\zeta t$, whose image under $\delta$ is $\tilde\zeta \rm{d}t$, which restricts to $0$ on $\tZ^{p,\rm{gst}}\times_{\tZp}T=\{\tilde\zeta=0\}$. So the action is trivial, and the coarse moduli cone is precisely $\Spec_{T^{\rm{gst}}}\operatorname{Sym}^{\bullet} I/I^2$, which is a line bundle supported on $\widetilde{\Xi}\times_{\tZp}T$ and trivial otherwise. By gluing different charts we get the line bundle $\tilde F$ on $\widetilde{\Xi}$.

The last part of the statement, namely that $\tilde F$ descends to a line bundle $F$ on $\Xi$ is proved by homogeneity: the normal cone of $\tZp/\tZ$ is homogeneous with respect to the $\Gm$-action with character $1$ on the fibers of $\tilde{V}_2\to \tZ^{p,\rm{gst}}$, but being a cone it is $0$-homogeneous as well (see the above discussion of homogeneity), so it is $\bar\gamma^*F$ for some line bundle $F$ on $\Xi\subseteq \tZ$.

\subsection*{Step V: the boundary pushes forward to zero} Recall that we need to show that the degree of the following class is $0$:
\[
 s_{\tilde{V}^{\rm{cpt}}}^![\overline {C^{\rm{gst}}}]=s_{\tilde{V}_2^{\rm{cpt}}}^!\circ s_{\tilde{V}_1^{\rm{cpt}}}^![N_{\overline {C^{\rm{gst}}}\cap (0\oplus {\tilde{V}_2^{\rm{cpt}}})}\overline {C^{\rm{gst}}}]
\]

It follows from the previous section that $s_{\tilde{V}_1^{\rm{cpt}}}^![N_{\overline {C^{\rm{gst}}}\cap (0\oplus {\tilde{V}_2^{\rm{cpt}}})}\overline {C^{\rm{gst}}}]$ can be represented by the sum of two cycles, one (call it $N_1$) supported on a line subbundle of $\tilde{V}_2^{\rm{cpt}}$ on $\widetilde\Xi$, the other one (call it $N_2$) supported on the zero section of $\tilde{V}_2^{\rm{cpt}}$.

\begin{lem}\label{lem:boundary}
 Both $N_1$ and $N_2$ are $5d+1$-dimensional cycles, and for $i=1,2$: \[\deg(s_{\tilde{V}_2^{\rm{cpt}}}^![N_i])=0.\]
\end{lem}
\begin{proof}
 Compare with \cite[Lemma 8.1]{CLpfields}. The dimension of $\tZ^{\rm{gst}}$ is $5d+3$, being locally a $5(d+1)$ vector bundle over a dimension $-2$ stack; so $\tZ^{p,\rm{gst}}$, which is a line bundle on the former, has dimension $5d+4$. The coarse moduli cone has then dimension $5d+6$, as can be argued from Lemma \ref{lem:open_cones}. $\tilde{V}_1^{\rm{cpt}}\rvert_{\oZp}$ has rank $5$, so $s_{\tilde{V}_1^{\rm{cpt}}}^![\overline {C^{\rm{gst}}}]$ is represented by a cycle of dimension $5d+1$. We shall exploit the commutativity of Gysin pullback with proper pushforward.

For $N_1$ we conclude from the above facts, since $\deg(s_{\tilde{V}_2^{\rm{cpt}}}^![N_1])=\deg(s_{V_2}^!\bar{\gamma}_*[N_1])$, but $\bar{\gamma}_*[N_1]\in A_{5d+1}(F)$ must be trivial, since $F$ has dimension $5d$, being a line bundle on $\Xi$ which is a divisor in $\Z^{\rm{main}}$.

On the other hand $N_2\subseteq \tilde{V}_2^{\rm{cpt}}$ admits a further splitting into $N_{2,\mu}\subseteq \tilde{V}_{2,\mu}^{\rm{cpt}}$ according to the component $D^{\mu}\oZp$ on which they are supported, with $\mu\vdash d$ in $k$ parts. There exists a comparison morphism:
\[\beta_{\mu}\colon D^\mu\oZp\to D^\mu\Z=D^\mu\oM_1(\PP^4,d)\to \pazocal W_{\mu}\]
where $\pazocal W_{\mu}:=\prod_{i=1}^k \M{0}{1}{\PP^4}{d_i}\times_{(\PP^r)^k}\PP^r$. The map $\beta_{\mu}$ is given by forgetting the $p$-field,  the Vakil-Zinger blow-up and the $k$-pointed elliptic curve contracted by the map to $\PP^4$. This has the nice property that $\tilde{V}_2^{\rm{cpt}}$ is the pullback along $\beta_{\mu}$ of a vector bundle on $\pazocal W_{\mu}$. First construct a connected curve $\overline \cC_\mu$ by gluing the universal curve over each factor along the given sections, producing a genus $0$ non-Gorenstein (unless $k=2$) singularity to which the universal maps to $\PP^4$ descend by the property of pushouts:
\bcd
\overline \cC_\mu \ar[d,"\bar\pi"]\ar[r,"\bar f"] & \PP^4 \\
\pazocal W_\mu &
\ecd
Notice now that the sheaf $V_\mu:=\bar\pi_*\bar f^*\OO_{\PP^4}(5)$ is a vector bundle of rank $5d+1$ on $\pazocal W_\mu$, as can be checked by Riemann-Roch and the normalisation sequence, and $\tilde{V}_{2,\mu}^{\rm{cpt}}=\beta_{\mu}^*(V_\mu^\vee)$. Finally the actual dimension of $\pazocal W_\mu$ is $5d+4-2k <5d+1$ for $k\geq 2$ by Kleiman-Bertini theorem, so $\beta_{\mu,*}[N_{2,\mu}]=0$.

\end{proof}
\begin{proof}(\emph{Main Theorem})
For the benefit of the reader we summarise here the key results which allow us to conclude:
\begin{itemize}
\item $p$-fields give the same invariants as the quintic up to a sign, see Theorem \ref{thm:p-fields-quintic}: \[\deg[\Mone{1}{\PP^4}{d}^p]^{\rm{vir}}_{\rm{loc}}= (-1)^{5d}\deg[\Mone{1}{X}{d}]^{\rm{vir}}.\]
\item $\Zp$ and $\Mone{1}{\PP^4}{d}^p$ are virtually birational, see Corollary \ref{cor:Zp-vs-pfields}:
\[\deg\virloc{\Zp}=\deg\virloc{\Mone{1}{\PP^4}{d}^p}.\]
\item The desingularisation $\tZp\to \Zp$ does not alter the invariants, see Lemma \ref{lem:tilding} \[\deg\virloc{\tZp}=\deg\virloc{\Zp}.\]
\item The main component of $\tZp$ contributes with the reduced invariants up to a sign, while the boundary is numerically irrelevant, see Lemmas \ref{lem:main-compo} and \ref{lem:boundary}:
\[\deg\virloc{\tZp}=(-1)^{5d}\deg\left(c_{\rm{top}}(\tilde\pi_*\tilde f^*\OO_{\PP^4}(5)\cap[\widetilde{\pazocal M}_1(\PP^4,d)^{\rm{main}}]\right).\]
\end{itemize}

\end{proof}

\newpage

\noindent Luca Battistella\\
Department of Mathematics, Imperial College London \\
\texttt{l.battistella14@imperial.ac.uk}\\

\noindent Francesca Carocci \\
Department of Mathematics, Imperial College London \\
\texttt{f.carocci14@imperial.ac.uk}\\

\noindent Cristina Manolache \\
Department of Mathematics, Imperial College London \\
\texttt{c.manolache@imperial.ac.uk}

\end{document}